\documentclass[11pt,a4paper]{amsart}
\usepackage{mathrsfs}
\usepackage{syntonly}
\usepackage{amsmath}
\usepackage{amsthm}
\usepackage{amsfonts}
\usepackage{amssymb}
\usepackage{latexsym}
\usepackage{amscd,amssymb,amsopn,amsmath,amsthm,graphics,amsfonts,mathrsfs,accents,enumerate,verbatim,calc}
\usepackage[dvips]{graphicx}
\usepackage[colorlinks=true,linkcolor=red,citecolor=blue]{hyperref}
\usepackage[all]{xy}
\usepackage{enumitem}

\date{}
\hypersetup{
colorlinks=blue,
linkcolor=black,
citecolor=red
}
\allowdisplaybreaks[4] \footskip=15pt
\renewcommand{\uppercasenonmath}[1]{}

\numberwithin{equation}{section} \theoremstyle{plain}
\newtheorem*{thm*}{Main Theorem}
\newtheorem{thm}{Theorem}[section]
\newtheorem{cor}[thm]{Corollary}
\newtheorem*{cor*}{Corollary}
\newtheorem{lem}[thm]{Lemma}
\newtheorem*{lem*}{Lemma}
\newtheorem{prop}[thm]{Proposition}
\newtheorem*{prop*}{Proposition}
\newtheorem{rem}[thm]{Remark}
\newtheorem*{rem*}{Remark}
\newtheorem{exa}[thm]{Example}
\newtheorem*{exa*}{Example}
\newtheorem{df}[thm]{Definition}
\newtheorem*{df*}{Definition}

\newtheorem*{conj*}{Conjecture}

\newtheorem*{ack*}{ACKNOWLEDGEMENTS}

\newcommand{\pf}{\noindent\begin {proof}}
\newcommand{\epf}{\end{proof}}
\newcommand{\bc}{\begin{center}}
\newcommand{\ec}{\end{center}}
\newcommand{\ccc}[3]{{
  \left(\begin{smallmatrix} {#1}   \\   {#2} \\\end{smallmatrix}\right)_{#3}}}

\newcommand{\ra}{\rightarrow}
\newcommand{\Ex}{\mbox{\rm Ext}}
\newcommand{\X}{\mathcal{X}}
\newcommand{\W}{\mathcal{W}}
\newcommand{\Mod}{\mbox{\rm Mod}}
\newcommand{\Hom}{\mbox{\rm Hom}}

\begin{document}
\begin{center}
{\Large  \bf  Recollements induced by left Frobenius pairs}

\vspace{0.5cm}  Yajun Ma, Dandan Sun, Rongmin Zhu and Jiangsheng Hu$\footnote{Corresponding author}$
\\
\medskip
\end{center}

\bigskip
\centerline { \bf  Abstract}
 \leftskip10truemm \rightskip10truemm \noindent
Given a right exact functor from an abelian category into another abelian category, there is an associated
abelian category called the comma category of the functor.
 In this paper, we characterize when left Frobenius pairs (resp. strong left Frobenius pairs) in abelian categories can induce left Frobenius pairs (resp. strong left Frobenius pairs) in their comma categories. This leads to the construction of recollements of right triangulated categories (resp. triangulated categories) from the stable categories of left Frobenius pairs (resp. strong left Frobenius pairs). Applications are given to complete hereditary cotorsion pairs and Gorenstein projective objects.
\leftskip10truemm \rightskip10truemm \noindent
\hspace{0.5em} \\[2mm]

\leftskip0truemm \rightskip0truemm

{\footnotesize\tableofcontents\label{contents}}

\section{Introduction}
Recollements of triangulated categories, introduced by Beilinson, Bernstein and Deligne in \cite{B}, have been
successfully applied to algebra, geometry, higher algebraic $K$-theory and so on. In the recent ten years,
there are a lot of interesting work on the constructions of recollements of (algebraic) triangulated categories such as derived categories of ordinary rings or differential graded rings, or more generally, stable
categories of Frobenius categories. This particularly includes the constructions of recollements of derived module categories from infinitely generated tilting modules and exact contexts (see \cite{Chen2012,Chen2021}) and the
ones of recollements of stable categories of monomorphism categories (see \cite{W,XZZ}).

In terms of relative generators and cogenerators in approximation theory, Becerril and coauthors  \cite{BMPS} introduced left Frobenius pairs and strong left Frobenius pairs, which are closely linked to the construction of exact model structures. Detailed definitions can be found in \cite[Definition 2.5]{BMPS}.
For any left Frobenius pair $(\mathcal{X},\mathcal{W})$, its stable category $\mathcal{X}/\mathcal{W}$ has a right triangulated structure which was introduced and studied by Beligiannis and Marmaridis in \cite{BM}, and if in addition $(\mathcal{X},\mathcal{W})$ is a strong left Frobenius pair, then its stable category $\mathcal{X}/\mathcal{W}$ has a triangulated structure (see Lemma \ref{lem:3.4} and Remark \ref{remark}).
A typical example is that for any Artin algebra $A$, the category of finitely generated left $A$-modules respect to the subcategory $\mathrm{inj}(A)$ of the injective objects form a left Frobenius pair, and it is a strong left Frobenius pair if and only if $A$ is selfinjective. More generally, left Frobenius pairs can be provided by complete hereditary cotorsion pairs of abelian categories in cotorion theory (see Lemma \ref{cotorsion pair}).

The purpose of the present paper is to give the construction of recollements of right triangulated categories (a generalization of triangulated categories) by addressing left Frobenius pairs in abelian categories.
To achieve this aim, we first give a general method to construct left Frobenius pairs of comma categories associated with right exact functors over abelian categories.

To state our main result more precisely, let us first introduce some definitions and notations.

$\bullet$ Let $\mathcal {X}$ be a subcategory of an abelian category $\mathscr{A}$. For convenience, we set
\begin{center}
{\vspace{2mm}
$\mathcal {X}^{\bot}:=\{M:\mbox{Ext}^i_\mathscr{A}(X,M)=0\mbox{ for every } X\in\mathcal {X} \mbox{~and all}~i\geq 1\}$;

\vspace{2mm}
$^{\bot}\mathcal {X}:=\{M:\mbox{Ext}^i_\mathscr{A}(M,X)=0\mbox{ for every } X\in\mathcal {X} \mbox{~and all} ~i\geq 1\}$.}
\end{center}

$\bullet$ Let $\mathcal{X}$ and $\mathcal{W}$ be two subcategories of an abelian category $\mathscr{A}.$
Then $\mathcal{W}$ is called a \emph{cogenerator} for $\mathcal{X}$ if $\mathcal{W}\subseteq \mathcal{X}$ and for any $X\in \mathcal{X}$ there exists an exact sequence
$0\ra X\ra W\ra X'\ra 0$ in $\mathscr{A}$ with $W\in \mathcal{W}$ and $X'\in \mathcal{X}.$
Moreover, $\mathcal{W}$ is called an \emph{injective cogenerator} for $\mathcal{X}$ if $\mathcal{W}\subseteq\mathcal{X}^{\perp}$ and $\mathcal{W}$ is a cogenerator for $\mathcal{X}$. Dually, we have the notions of generator and projective generator.

$\bullet$ Let $\mathscr{A}$ be an abelian category. Recall from \cite[Definition 2.5]{BMPS} that a pair $(\X,\W)$ is called a left \emph{Frobenius pair} in $\mathscr{A}$ if the following conditions are satisfied:
\begin{enumerate}
 \item $\X$ is closed under extensions, kernels of epimorphisms and direct summands;

\item $\mathcal{W}$ is closed under direct summands in $\mathscr{A}$;

\item $\mathcal{W}$ is an injective cogenerator for $\X$.
 \end{enumerate}

\noindent If in addition $\W$ is also a projective generator for $\X$, then we say that $(\X,\W)$ is a \emph{strong left Frobenius pair}.

$\bullet$ Recall that for any abelian categories $\mathscr{A}$ and $\mathscr{B}$  and any right exact functor $T:\mathscr{B}\rightarrow\mathscr{A}$,  there exists an abelian category, denoted by $(T\downarrow\mathscr{A})$, consists of objects all triples $\left(\begin{smallmatrix}A\\B\end{smallmatrix}\right)_{\varphi}$ where $\varphi: T(B)\rightarrow A$ is a morphism in $\mathscr{A}$. We note that this new abelian category is called by \emph{comma category} in \cite{FGR,Marmaridis}. Detail definitions can be found in Definition \ref{df:comma} below.
Examples of comma categories include but are not limited to: the category of modules or complexes over a triangular matrix ring, morphism category of an abelian category and so on (see Example \ref{exact} below).

$\bullet$ Let $\mathcal{X}$ be a subcategory of an abelian category $\mathscr{A}$.
We denote by $\widehat{\mathcal{X}}$ the class of objects $C$ such that
there exists an exact sequence $0\ra X_{n}\ra X_{n-1}\ra\cdots\ra X_{0}\ra C\ra 0$ with each $X_{i}\in \mathcal{X}$ for some $n$.

$\bullet$ Let $\mathcal{X}$ be a  subcategory of $\mathscr{A}$ and $\mathcal{Y}$ a subcategory of $\mathscr{B}.$ We need the following subcategory of $(T\downarrow\mathscr{A}),$ which was inspired by \cite{LZ}.
\begin{displaymath}
\mathfrak{B}^{\mathcal{X}}_{\mathcal{Y}}:=\{\left(\begin{smallmatrix}X\\Y\end{smallmatrix}\right)_{\varphi}\in (T\downarrow\mathscr{A})~|~ Y\in\mathcal {Y}, \ \varphi \ \mbox{is monic and Coker}\varphi\in\mathcal {X}\}.
\end{displaymath}

$\bullet$ Suppose that $T:\mathscr{B}\rightarrow \mathscr{A}$ is a right exact functor and $\mathscr{B}$ has enough projective objects. Following \cite{CE1956}, we can construct a left derived functor $\mathrm{L}_{n}T$ for any $n\geq 1$.
For an object $B$ in $\mathscr{B}$, one can choose a projective resolution of $B$
$$\mathbf{P}:\cdots\rightarrow P_{n+1}\xrightarrow{d_{n}} P_{n}\xrightarrow{d_{n-1}} P_{n-1}\rightarrow \cdots \xrightarrow{d_{1}} P_{1}\xrightarrow{d_{0}} P_{0}\rightarrow B\rightarrow 0$$
and define $(\mathrm{L}_{n}T)B=\mathrm{Ker}Td_{n-1}/\mathrm{Im}Td_{n}.$

From now on, let $T:\mathscr{B}\rightarrow \mathscr{A}$ be a right exact functor from
an abelian category $\mathscr{B}$ to another abelian category $\mathscr{A}$, where $\mathscr{B}$ is supposed to have enough projective objects.

Now, our first main result can be stated as follows.

\begin{thm}\label{thm:1.1} Let $(T\downarrow\mathscr{A})$ be the comma category of the functor $T$. Assume that $\mathcal{W}$ and $\mathcal{X}$ are subcategories of $\mathscr{A}$ and that $\mathcal{V}$ and $\mathcal{Y}$ are subcategories of $\mathscr{B}.$ If $T(\mathcal{V})\subseteq \widehat{\mathcal{W}}$ and $(\mathrm{L}_{n}T)\mathcal{Y}=0$ for all $n\geq 1$,
then the following statements hold.
\begin{enumerate}
 \item[{\rm (1)}] The pair $(\mathfrak{B}^{\mathcal{X}}_{\mathcal{Y}},\mathfrak{B}^{\mathcal{W}}_{\mathcal{V}})$ is a left Frobenius pair in $(T\downarrow\mathscr{A})$ if and only if $(\mathcal{X},\mathcal{W})$ and $(\mathcal{Y},\mathcal{V})$ are left Frobenius pairs in $\mathscr{A}$ and $\mathscr{B}$, respectively.

 \item[{\rm (2)}] If $T(\mathcal{Y})\subseteq\widehat{\mathcal{X}}$, then the pair $(\mathfrak{B}^{\mathcal{X}}_{\mathcal{Y}},\mathfrak{B}^{\mathcal{W}}_{\mathcal{V}})$ is a strong left Frobenius pair in $(T\downarrow\mathscr{A})$ if and only if $(\mathcal{X},\mathcal{W})$ and $(\mathcal{Y},\mathcal{V})$ are strong left Frobenius pairs in $\mathscr{A}$ and $\mathscr{B}$, respectively.

 \end{enumerate}
\end{thm}

As a consequence of Theorem \ref{thm:1.1}, we obtain  recollements of (right) triangulated categories from the stable categories of left Frobenius pairs.

\begin{thm}\label{thm:1.2} Let $(T\downarrow\mathscr{A})$ be the comma category of the functor $T$, let $\mathcal{W}$ and $\mathcal{X}$ be subcategories of $\mathscr{A}$, and let $\mathcal{V}$ and $\mathcal{Y}$ be subcategories of $\mathscr{B}$. Assume that $T(\mathcal{V})\subseteq \mathcal{W}$, $T(\mathcal{Y})\subseteq \mathcal{X}$ and  $(\mathrm{L}_{n}T)\mathcal{Y}=0$ for all $n\geq 1$.
\begin{enumerate}
\item[{\rm (1)}] If $(\mathcal{X},\mathcal{W})$ and $(\mathcal{Y},\mathcal{V})$ are left Frobenius pairs in $\mathscr{A}$ and $\mathscr{B}$, respectively, then there exists a recollement of right triangulated categories

\vspace{2mm}
$$\scalebox{0.85}{\xymatrixcolsep{3pc}\xymatrix{
  \mathcal{X}/\mathcal{W}\ar[rr] && \mathfrak{B}^{\mathcal{X}}_{\mathcal{Y}}/\mathfrak{B}^{\mathcal{W}}_{\mathcal{V}}\ar[rr] \ar@/^1.8pc/[ll]\ar@/_1.8pc/[ll] && \mathcal{Y}/\mathcal{V}, \ar@/^1.8pc/[ll]\ar@/_1.8pc/[ll]}}$$
\vspace{2mm}

\noindent where all the functors can be described explicitly.
\item [\rm (2)]  If $(\mathcal{X},\mathcal{W})$ and $(\mathcal{Y},\mathcal{V})$ are strong left Frobenius pairs in $\mathscr{A}$ and $\mathscr{B}$, respectively, then the recollement in $(1)$ is a recollement of triangulated categories.
\end{enumerate}
\end{thm}

Let us remark that left Frobenius pairs and comma categories are far generalizations of module categories respect to injective modules and module categories of triangular matrix rings. Our Theorem \ref{thm:1.2} provides a proper framework for  constructions of recollements by addressing left Frobenius pairs.
If we apply Theorem \ref{thm:1.2} to modules categories respect to injective modules and triangular matrix Artin algebras, then we recover \cite[Theorem 1.3]{XZZ}.

As pointed out in \cite{W}, Wang and Liu constructed recollements of triangulated categories from the stable categories of Frobenius categories in the context of comma categories associated with fully faithful exact functors over abelian categories. It should be pointed out that our Theorem \ref{thm:1.2} offers a different perspective. More precisely, we first construct left Frobenius pairs in the comma category $(T\downarrow\mathscr{A})$ from left Frobenius pairs in $\mathscr{A}$ and $\mathscr{B}$ under the mild condition on $T$, then construct recollements of triangulated categories from the stable categories of strong left Frobenius pairs. In fact, if $(\mathcal{X},\mathcal{W})$ is a strong left Frobenius pair, then $\mathcal{X}$ is a Frobenius category with projective-injective objects $\mathcal{W}$, and the converse is not true in general. As $T$ in our setting is not supposed to be a fully faithful exact functor, our Theorem \ref{thm:1.2}(2) cannot be recovered by \cite[Theorem 1.1]{W}. Moreover, it follows from Lemma \ref{cotorsion pair} that there are a large variety of examples of left Frobenius pairs which are not necessary Frobenius categories. Thanks to Theorem \ref{thm:1.2}(1), we can construct many recollements of right triangulated categories which cannot be obtained by \cite{W}.

The contents of this paper are arranged as follows.
In Section \ref{section2}, we give some terminologies and some preliminary results which are need for our proof.
In Section \ref{section3}, we study the construction of (strong) left Frobenius pairs in the comma category $(T\downarrow \mathscr{A})$ from (strong) left Frobenius pairs in $\mathscr{A}$ and $\mathscr{B}$, including Theorem \ref{thm:1.1}. In Section \ref{section4}, we first employ Theorem \ref{thm:1.1} to prove Theorem \ref{thm:1.2}, and then establish recollements of (right) triangulated categories by applying Theorem \ref{thm:1.2} to complete hereditary cotorsion pairs and Gorenstein projective objects.

\section{Preliminaries}\label{section2}
Throughout this paper, $\mathscr{A}$ and $\mathscr{B}$ are abelian categories with enough projective and injective objects. For any ring $R$, $R\text{-}\Mod~(\mathrm{resp}.~R\text{-}{\rm mod})$ is the category of  (resp. finitely generated) left $R$-modules and $\textrm{Ch}(R)$ is the category of complexes of left $R$-modules.
We use $_RM$ (resp. $M_R$) to denote a left (resp. right) $R$-module $M$, and the  projective, injective and flat dimensions of $_RM$ (resp. $M_R$) will be denoted by $\mathrm{pd}_RM$, $\mathrm{id}_RM$ and $\mathrm{fd}_RM$ (resp. $\mathrm{pd} M_R$, $\mathrm{id} M_R$ and $\mathrm{fd} M_R$), respectively. For unexplained ones, we refer the reader to \cite{FGR,Gober}.
\subsection{Comma categories}
Now we recall some notions and facts of comma categories, which can be found in \cite{FGR} and \cite{Marmaridis}.

\begin{df}\label{df:comma}\cite{FGR,Marmaridis}
Let $T:\mathscr{B}\rightarrow\mathscr{A}$ be a right exact functor. Then the \emph{comma category }$(T\downarrow\mathscr{A})$ is defined as follows:
 \begin{enumerate}
\item [{\rm (1)}] The objects are triples $\left(\begin{smallmatrix}A\\B\end{smallmatrix}\right)_{\varphi}$, with $A\in{\mathcal{\mathscr{A}}}$, $B\in{\mathcal{\mathscr{B}}}$ and $\varphi: T(B)\rightarrow A$ being a morphism in $\mathscr{A}$;

\item [{\rm (2)}] A morphism $\left(\begin{smallmatrix}a\\b\end{smallmatrix}\right):\left(\begin{smallmatrix}A\\B\end{smallmatrix}\right)_{\varphi}\rightarrow\left(\begin{smallmatrix}A'\\B' \end{smallmatrix}\right)_{\varphi'}$ is given by two morphisms $a:A\rightarrow A'$ in $\mathscr{A}$ and $b: B\rightarrow B'$ in $\mathscr{B}$ such that $\varphi' T(b)=a\varphi$.

 \end{enumerate}
 \indent If there is no possible confusion, we sometimes omit the morphism $\varphi$. Recall from \cite{FGR} that the comma category $(T\downarrow\mathscr{A})$ is indeed
an abelian category since the functor T is assumed to be right exact.
 \end{df}

\begin{rem}\label{remark:recollement}
For a comma category $(T\downarrow\mathscr{A}),$ we have the following functors.
\begin{enumerate}
\item [{\rm (1)}] The functor $T_{\mathscr{B}}:\mathscr{B}\ra (T\downarrow\mathscr{A})$ is defined on objects $Y\in \mathscr{B}$ by $T_{\mathscr{B}}(Y)=\ccc{T(Y)}{Y}{}$ and given a morphism $\beta:Y\ra Y'$ in $\mathscr{B}$ then $T_{\mathscr{B}}(\beta)=\ccc{T(\beta)}{\beta}{}$ is a morphism in $(T\downarrow\mathscr{A}).$

\item [{\rm (2)}] The functor $U_{\mathscr{B}}:(T\downarrow\mathscr{A})\ra \mathscr{B}$ is defined on objects  $\ccc{X}{Y}{\varphi}\in{(T\downarrow\mathscr{A})}$ by $U_{\mathscr{B}}(\ccc{X}{Y}{\varphi})=Y$ and given a morphism $\ccc{\alpha}{\beta}{}:\ccc{X}{Y}{\varphi}\ra \ccc{X'}{Y'}{\varphi'}$ in $(T\downarrow\mathscr{A})$ then $U_{\mathscr{B}}(\ccc{\alpha}{\beta}{})=\beta.$ Similarly, we can define the functor $U_{\mathscr{A}}:(T\downarrow\mathscr{A})\ra \mathscr{A}.$

\item [{\rm (3)}] The functor $Z_{\mathscr{B}}:\mathscr{B}\ra (T\downarrow\mathscr{A})$ is defined on objects $Y\in \mathscr{B}$ by $T_{\mathscr{B}}(Y)=\ccc{0}{Y}{}$ and given a morphism $\beta:Y\ra Y'$ in $\mathscr{B}$ then $Z_{\mathscr{B}}(\beta)=\ccc{0}{\beta}{}$ is a morphism in $(T\downarrow\mathscr{A}).$ Similarly, we can define the functor $Z_{\mathscr{A}}:\mathscr{A}\ra (T\downarrow\mathscr{A}).$

\item [{\rm (4)}] The functor $q:(T\downarrow\mathscr{A})\ra \mathscr{A}$ is defined on objects $\ccc{X}{Y}{\varphi}\in{(T\downarrow\mathscr{A})}$ by $q(\ccc{X}{Y}{\varphi})=\mathrm{Coker}\varphi$ and given a morphism $\ccc{\alpha}{\beta}{}:\ccc{X}{Y}{\varphi}\ra \ccc{X'}{Y'}{\varphi'}$ in $\mathscr{B}$ then $q(\ccc{\alpha}{\beta}{}):\mathrm{Coker}\varphi\ra \mathrm{Coker}\varphi'.$
 \end{enumerate}
 It is easy to check, see also \cite{Psaroudakis,Psaroudakis2018} that the above data define the following recollement of abelian categories
$$\scalebox{0.85}{\xymatrixcolsep{3pc}\xymatrix{
  \mathscr{A}\ar[rr]^{Z_{\mathscr{A}}} && (T\downarrow \mathscr{A})\ar[rr]^{U_{\mathscr{B}}} \ar@/^1.8pc/[ll]^{U_{\mathscr{A}}}\ar@/_1.8pc/[ll]_{q} && \mathscr{B}. \ar@/^1.8pc/[ll]^{Z_{\mathscr{B}}}\ar@/_1.8pc/[ll]_{T_{\mathscr{B}}}}}$$
\end{rem}

Next we give some examples of comma categories.
\begin{exa}\label{exact}{\rm
(1) Let $R$ and $S$ be two rings, ${_R}M_S$ an $R$-$S$-bimodule, and $\Lambda=\left(
                                                                                 \begin{smallmatrix}R & M \\0 & S\\\end{smallmatrix}
                                                                               \right)
$ the triangular matrix ring. If we define $T\cong M\otimes_{S}-: S\text{-}\Mod \ra R\text{-}\Mod$, then we get that ${\Lambda}\text{-}\Mod$ is equivalent to the comma category $(T\downarrow R\text{-}\Mod )$.

(2) Let $\Lambda=\left(\begin{smallmatrix}R & M \\0 & S\\\end{smallmatrix}\right)
$ be a triangular matrix ring.  If we define  {\rm$T\cong M\otimes_{S}-: \textrm{Ch}(S) \ra \textrm{Ch}(R)$}, then {\rm$\textrm{Ch}(\Lambda)$} is equivalent to the comma category $(T\downarrow {\rm\textrm{Ch}}(R))$.

(3) If $\mathscr{A}=\mathscr{B}$ and $T$ is the identity functor, then the comma category $(T\downarrow \mathscr{A})$ coincides with the
morphism category ${\rm mor}(\mathscr{A})$ of $\mathscr{A}$.

(4) Let  $\mathscr{A}=R$-$\mathrm{Mod}$ and $\mathscr{B}=$Ch$(R)$. If we define $e: \mathscr{B}\rightarrow\mathscr{A}$ via $C^\bullet\mapsto C^0$ for any $C^\bullet\in{\mathscr{B}}$, then $e$ is an exact functor and we have a comma category $(e\downarrow\mathscr{A})$.
}\end{exa}

\subsection{Cotorsion pairs}

Recall that a \emph{cotorsion pair} in $\mathscr{A}$ is a pair of classes $(\mathcal{X},\mathcal{Y})$ of objects in $\mathscr{A}$ satisfying
the following two conditions:
(i) $X\in \mathcal{X}$ if and only if $\Ex^{1}
_{\mathscr{A}} (X,Y)=0$ for all $Y\in\mathcal{Y}$;
(ii) $Y\in \mathcal{Y}$ if and only if $\Ex^{1}
_{\mathscr{A}}(X,Y)=0$ for all $X\in \mathcal{X}$.
The cotorsion pair is called \emph{complete} if, for any $A\in \mathscr{A}$, there exists an exact sequence
$0\rightarrow Y\rightarrow X\rightarrow A\rightarrow 0$ with $X\in \mathcal{X}$ and $Y\in \mathcal{Y}$, and another exact sequence $0 \rightarrow A\rightarrow Y'
\rightarrow X'\rightarrow 0$ with $X'\in \mathcal{X}$ and $Y'\in \mathcal{Y}$.
We say that the cotorsion pair $(\mathcal{X},\mathcal{Y})$ is \emph{hereditary} if $\Ex^{i\geq 1}_{\mathscr{A}}(\mathcal{X},\mathcal{Y})=0.$
By a \emph{projective cotorsion pair} in $\mathscr{A}$, we
mean a complete cotorsion pair $(\mathcal{C},\mathcal{W})$ for which $\mathcal{W}$ is thick (meaning it satisfies the 2 out of 3 property on short exact sequences) and $\mathcal{C}\cap \mathcal{W}$ is the class of
projective objects; see \cite{Gillespie2016}. It is easy to check that every projective cotorsion pair is hereditary.
The book \cite{EJ2} is a standard reference for cotorsion pairs.

The follow result indicates that left Frobenius pairs can be provided
by complete hereditary cotorsion pairs of abelian categories in cotorion theory.
\begin{lem}\label{cotorsion pair}
Let $(\mathcal{X},\mathcal{Y})$ be a complete hereditary cotorsion pair in $\mathscr{A}$. Then
\begin{enumerate}
\item [\rm{(1)}] $(\mathcal{X},\mathcal{X}\cap\mathcal{Y})$ is a left Frobenius pair in $\mathscr{A}$.
\item [\rm{(2)}] $(\mathcal{X},\mathcal{X}\cap\mathcal{Y})$ is a  strong left Frobenius pair in $\mathscr{A}$ if and only if $(\mathcal{X},\mathcal{Y})$ is a projective cotorsion pair.
\end{enumerate}
\end{lem}
\begin{proof}
(1). Since $(\mathcal{X},\mathcal{Y})$ is a complete hereditary cotorsion pair, it is easy to check that $\mathcal{X}$ is closed under extensions, kernels of epimorphisms and direct summands, and $\mathcal{X}\cap\mathcal{Y}$ is closed under direct summands.
Let $X$ be in $\mathcal{X}$.
Then there exists an exact sequence
$0\ra X\ra Y\ra X'\ra 0$ with $Y\in \mathcal{Y}$ and $X'\in \mathcal{X}$.
Since $\mathcal{X}$ is closed under extensions, $Y\in \mathcal{X}$.
So $\mathcal{X}\cap\mathcal{Y}$ is a cogenerator for $\mathcal{X}.$
Note that $\mathrm{Ext}_{\mathscr{A}}^{i\geq 1}(\mathcal{X},\mathcal{X}\cap\mathcal{Y})=0$.
Thus $(\mathcal{X},\mathcal{X}\cap\mathcal{Y})$ is a left Frobenius pair in $\mathscr{A}$.

(2) $``\Rightarrow"$. Let $P$ be a projective object of $\mathscr{A}$.
Then $P\in \mathcal{X}$ as $(\mathcal{X},\mathcal{Y})$ is a cotorsion pair.
Since $(\mathcal{X},\mathcal{X}\cap\mathcal{Y})$ is a left Frobenius pair in $\mathscr{A}$, we have an exact sequence
$0\ra X\ra W\ra P\ra 0$ with $W\in \mathcal{X}\cap\mathcal{Y}$ and $X\in \mathcal{X}.$
Thus $W\cong X\oplus P.$
Hence $P\in \mathcal{X}\cap\mathcal{Y}$ as $\mathcal{X}\cap\mathcal{Y}$ is closed under direct summands.
For any $Z\in \mathcal{X}\cap\mathcal{Y},$ there exists an exact sequence
$0\ra X'\ra P\ra Z\ra 0$ with $P$ projective as $\mathscr{A}$ has enough projective objects.
Since $\mathcal{X}$ is closed under kernels of epimorphisms, $X'\in \mathcal{X}.$
Thus $\mathrm{Ext}_{\mathscr{A}}^{1}(Z,X')=0.$
So $Z$ is projective.
It follows that $\mathcal{X}\cap\mathcal{Y}$ is the class of
projective objects.
Therefore $(\mathcal{X},\mathcal{Y})$ is a projective cotorsion pair by \cite[Proposition 3.7]{Gillespie2016}.

 $``\Leftarrow"$. Suppose that $(\mathcal{X},\mathcal{Y})$ is a projective cotorsion pair in $\mathscr{A}$, it suffices to show that $\mathcal{X}\cap\mathcal{Y}$ is a projective generator for $\mathcal{X}$.
Note that $\mathcal{X}\cap \mathcal{Y}$ is the class of projective objects. It follows that $\mathrm{Ext}_{\mathscr{A}}^{i\geq 1}(\mathcal{X}\cap\mathcal{Y},\mathcal{X})=0$.
Let $X\in \mathcal{X}.$ Since $\mathscr{A}$ has enough projective objects, there exists an exact sequence
$0\ra K\ra P\ra X\ra 0$ with $P$ projective.
Since $\mathcal{X}$ is closed under kernels of epimorphisms, $K\in \mathcal{X}.$
Thus $\mathcal{X}\cap\mathcal{Y}$ is a projective generator for $\mathcal{X}.$
\end{proof}

\subsection{Gorenstein objects}
Recall that  an object $M$ in $\mathscr{A}$ is called \emph{Gorenstein projective} if $M=\textrm{Z}^{0}(P^\bullet)$ for some exact complex $P^\bullet$ of projective objects which remains exact after applying ${\rm Hom}_{\mathscr{A}}(-,P)$ for any projective object $P$. Let $\mathcal{GP}_\mathscr{A}$ (resp. $\mathcal{P}_\mathscr{A}$) be the subcategory of $\mathscr{A}$ consisting of Gorenstein projective objects (resp. projective objects).  For any ring $A$, we denote by $\mathcal{GP}(A)$ (resp. $\mathcal{P}(A)$) the subcategory of $A\text{-}\Mod$ consisting of Gorenstein projective modules (resp. projective modules).
 We have that $(\mathcal{GP}_\mathscr{A},\mathcal{P}_\mathscr{A})$ is a strong left Frobenius pair and its proof is similar to that of Proposition 6.1 in \cite{BMPS}.

For any associative ring $R$ with identity,
a left $R$-module $N$ is called \emph{Gorenstein flat}  \cite{EJ2,Holm}~if there is an exact sequence of flat left $R$-modules
$$\mathbf{F}= \cdots\rightarrow F_{1}\rightarrow F_{0}\rightarrow F^{0}\rightarrow F^{1}\rightarrow\cdots$$
with $N=\mathrm{Ker}(F_{0}\rightarrow F^{0})$ such that $Q\otimes_{R}\mathbf{F}$ is exact for any injective right $R$-module $Q$.
  Let $\mathcal{GF}(R)$ be the full subcategory of $R$-$\mathrm{Mod}$ consisting of all Gorenstein flat modules.
 It is well known that $(\mathcal{F}(R),\mathcal{C}(R))$, where $\mathcal{F}(R)$ is the class of flat left $R$-modules and $\mathcal{C}(R)=\mathcal{F}(R)^{\perp}$ is the class of cotorsion left $R$-modules, is a complete hereditary cotorsion pair. Recently $\mathrm{\check{S}}$aroch and $\mathrm{\check{S}\acute{t}}$ov$\mathrm{\acute{\i}\check{c}}$ek have shown in \cite{JJarxiv} that $(\mathcal{GF}(R),\mathcal{GF}(R)^{\perp})$ is a complete hereditary cotorsion pair regardless of the ring $R$  and $\mathcal{GF}(R) \cap\mathcal{GF}(R)^{\perp}=\mathcal{F}(R)\cap \mathcal{C}(R).$
Thus $(\mathcal{GF}(R),\mathcal{F}(R)\cap \mathcal{C}(R))$ and $(\mathcal{F}(R),\mathcal{F}(R)\cap \mathcal{C}(R))$ are left Frobenius pairs by Lemma \ref{cotorsion pair}.

\subsection{Right triangulated categories}
We first introduce the definitions of right triangle functors and recollements of right triangulated categories.
\begin{df} \cite{BM,BR}
Let $\mathscr{C}$ be an additive category with an additive endofunctor $\Sigma: \mathscr{C} \rightarrow \mathscr{C}$.
Consider in $\mathscr{C}$ the category $\mathcal{RT}(\mathscr{C},\Sigma)$ whose objects are diagrams of the form
 $A\xrightarrow{f}B\xrightarrow{g}C\xrightarrow{h} \Sigma A$
with $A,B,C \in \mathscr{C}$ and morphisms indicated by the following commutative diagram:
$$\xymatrix{
A\ar[r]^f\ar[d]_\alpha&B\ar[r]^g\ar[d]_\beta&C\ar[r]^h\ar[d]_\gamma&\Sigma A\ar[d]_{\Sigma(\alpha)}\\
A'\ar[r]^{f'}&B'\ar[r]^{g'}&C'\ar[r]^{h'}&\Sigma A'
.}$$
A \emph{right triangulation} of the pair $(\mathscr{C},\Sigma)$ is a full subcategory $\Delta$ of $\mathcal{RT}(\mathscr{C},\Sigma)$ which satisfies all the axioms of a triangulated category, except that $\Sigma$ is not necessarily an equivalence.
Then the triple $(\mathscr{C},\Sigma,\Delta)$ is called a \emph{right triangulated category}, $\Sigma$ is the \emph{suspension functor}
and the diagrams in $\Delta$ are the \emph{right triangles}.
\end{df}
\begin{df}
A covariant functor $F:\mathscr{C}\ra \mathscr{C}'$ is said to be a right triangle functor from the right triangulated category $(\mathscr{C},\Sigma,\Delta)$ to the right triangulated category $(\mathscr{C}',\Sigma',\Delta')$, if $\varphi:F\Sigma\ra \Sigma'F$ is a natural isomorphism and for any right triangle $A\xrightarrow{f}B\xrightarrow{g}C\xrightarrow{h} \Sigma A$ of $\Delta$, then $F(A)\xrightarrow{F(f)}F(B)\xrightarrow{F(g)}F(C)\xrightarrow{\varphi_{A}F(h)} \Sigma' F(A)$ is a right triangle of $\Delta'$.
\end{df}

\begin{rem}
The right triangle functors between triangulated categories are precisely triangle functors between triangulated categories.
\end{rem}

\begin{df} A recollement situation
 between right triangulated categories $\mathscr{A}, \mathscr{C}$
and $\mathscr{B}$ is a diagram
~$$\scalebox{0.85}{\xymatrixcolsep{3pc}\xymatrix{
  \mathscr{A}\ar[rr]^{i} && \mathscr{C}\ar[rr]^{e} \ar@/^1.6pc/[ll]^{p}\ar@/_1.6pc/[ll]_{q} && \mathscr{B} \ar@/^1.6pc/[ll]^{r}\ar@/_1.6pc/[ll]_{l} } }$$
of right triangle functors satisfying the following conditions:
\begin{enumerate}
\item [\rm{(1)}] $(q,i)$ $(i,p)$, $(l,e)$ and $(e,,r)$ are adjoint pairs;

\item [\rm{(2)}] The functors $i$, $l$ and $r$ are fully faithful;

\item [\rm{(3)}] $\mathrm{Im}i=\mathrm{Ker}e$.
\end{enumerate}
\end{df}

In what follows, for any left Frobenius pair $(\mathcal{X},\mathcal{W})$, we always denote by $\overline{\mathcal{X}}=\mathcal{X}/\mathcal{W}$ the stable category of $\mathcal{X}$, the objects of $\overline{\mathcal{X}}$ are objects of $\mathcal{X}$, and $\Hom_{\overline{\mathcal{X}}}(X,Y)=\Hom_{\mathcal{X}}(X,Y)/\mathcal{W}(X,Y),$ where $\mathcal{W}(X,Y)$ is the subgroup of $\Hom_{\mathcal{X}}(X,Y)$ consisting of those morphisms factoring through objects of $\mathcal{W}.$

\begin{lem}\label{lem:3.4}
If $(\mathcal{X},\mathcal{W})$ is a left Frobenius pair in $\mathscr{A}$, then $\overline{\mathcal{X}}$  has a right triangulated structure.
\end{lem}
\begin{proof} By \cite[Theorem 2.12]{BM}, it suffices to show that for any morphism $f:X_{1}\rightarrow X_{2}$, if $\Hom_{\mathscr{A}}(f,W):\Hom_{\mathscr{A}}(X_{2},W)\rightarrow \Hom_{\mathscr{A}}(X_{1},W)$ is epic for any $W\in \mathcal{W}$, then $\mathrm{Coker} f\in{\mathcal{X}}$.

Note that $\mathcal{W}$ is an injective cogenerator for $\mathcal{X}$. Then
there exists an exact sequence $0\ra X_{1}\xrightarrow{g}Q\xrightarrow{h}X'\ra 0$ in $\mathscr{A}$ with $Q\in \mathcal{W}$ and $X'\in \mathcal{X}.$
Thus there exists a morphism $k:X_{2}\ra W$ with $g=kf$, and hence $f$ is monic as $g$ is monic.
By the exactness of the sequence $0\ra X_{1}\xrightarrow{f} X_{2}\xrightarrow{\pi}\mathrm{Coker} f\ra 0$, we have an exact sequence $\Hom_{\mathscr{A}}(X_{2},W)\rightarrow \Hom_{\mathscr{A}}(X_{1},W)\ra \Ex^{1}_{\mathscr{A}}(\mathrm{Coker} f,W)\ra 0$ for any $W\in{\mathcal{W}}$.
So $\mathrm{Coker} f\in \widehat{\mathcal{X}}\cap{^{\perp}\mathcal{W}}=\mathcal{X}$ by \cite[Proposition 2.14]{BMPS}, as desired.
\end{proof}

\begin{rem}\label{remark:right triangulated categories}
Let $(\mathcal{X},\mathcal{W})$ be a left Frobenius pair in $\mathscr{A}$.
For any $X\in\mathcal{X}$, we have an exact sequence $0\ra X\ra W\ra X'\ra 0$ in $\mathscr{A}$ with $W\in \mathcal{W}$ and $X'\in \mathcal{X}$. Define $\Sigma(X)=X'$ to be the image of $X'$ in $\overline{\mathcal{X}}$.
Then $\Sigma:\overline{\mathcal{X}}\ra \overline{\mathcal{X}}$ is an additive functor.
 Let $0\ra X_{1}\xrightarrow{f}X_{2}\xrightarrow{g}X_{3}\ra 0$ be an exact sequence in $\mathcal{X}$. For any exact sequence $0\ra X_{1}\xrightarrow{h}W_{1}\xrightarrow{l}\Sigma X_{1}\ra 0$ with $W_{1}\in\mathcal{W}$ and $\Sigma X_{1}\in \mathcal{X},$
there exists the following commutative diagram
$$\xymatrix{0\ar[r]&X_{1}\ar[r]^{f}\ar@{=}[d]&X_{2}\ar[r]^{g}\ar[d]^{a}&X_{3}\ar[r]\ar[d]^{b}&0\\
0\ar[r]&X_{1}\ar[r]^{h}&W_{1}\ar[r]^{l}&\Sigma X_{1}\ar[r]&0
}
$$
as $\Ex_{\mathscr{A}}^{1}(X_{3},W_{1})=0.$
Hence the sequence
$\xymatrix@C=0.6cm{X_{1}\ar[r]^{\overline{f}}&X_{2}\ar[r]^{\overline{g}}&X_{3}\ar[r]^{\overline{b}}&\Sigma X_{1}}$
in $\overline{\mathcal{X}}$ is called a \emph{distinguished right triangle}. The right triangles in $\overline{\mathcal{X}}$ are defined as the sequences which are isomorphic to some  distinguished right triangles obtained in this way by \cite{BM}.
\end{rem}
\begin{rem}\label{remark}
If $(\mathcal{X},\mathcal{W})$ is a strong left Frobenius pair in $\mathscr{A}$, then $\mathcal{X}$ is a Frobenius category with projective-injective objects $\mathcal{W}$.
Thus the stable category $\overline{\mathcal{X}}$ is triangulated by \cite{Happel1988}.
\end{rem}

\section{Left Frobenius pairs and comma categories}\label{section3}
In this section, we construct (strong) left Frobenius pairs in the comma category $(T\downarrow \mathscr{A})$ from (strong) left Frobenius pairs in abelian categories $\mathscr{A}$ and $\mathscr{B}$, which contains Theorem \ref{thm:1.1} in the introduction.

We assume that $\mathcal{W}$ and $\mathcal{X}$ are two subcategories of $\mathscr{A}$, and $\mathcal{V}$ and $\mathcal{Y}$ are two subcategories of $\mathscr{B}.$
Recall from the introduction that  $\mathrm{L}_{n}T$ is the \emph{nth left derived functor} of $T$ for any integer $n\geq 1$. The following Ext formulas for comma categories are essentially taken from \cite{FGR}.
They were also proved in a much more general context by \cite[Corollary 4.2]{B2000}.

\begin{lem}\label{lem:extension group}
Let $(T\downarrow \mathscr{A})$ be a comma category. Then the following statements hold.
\begin{enumerate}
\item [\rm{(1)}]If $(\mathrm{L}_{n}T)Y=0$ for all $n\geq 1,$
then $\Ex^{i}_{(T\downarrow\mathscr{A})}(\ccc{T(Y)}{Y}{},\ccc{X_{1}}{X_{2}}{\varphi})\cong\Ex^{i}_{\mathscr{B}}(Y,X_{2})$ for all $i\geq 1$.

\item [\rm{(2)}] $\Ex^{i}_{(T\downarrow\mathscr{A})}(\ccc{X}{0}{},\ccc{N_{1}}{N_{2}}{\varphi})\cong\Ex^{i}_{\mathscr{A}}(X,N_{1})$ for all $i\geq 1$.
\end{enumerate}
\end{lem}

Let $\mathcal{X}$ be a subcategory of $\mathscr{A}$. Recall from the introduction that
$\widehat{\mathcal{X}}$ is the class of objects $C$ such that
there exists an exact sequence $0\ra X_{n}\ra X_{n-1}\ra\cdots\ra X_{0}\ra C\ra 0$ with each $X_{i}\in \mathcal{X}$ for some $n$.

In the following lemma, we collect some results from \cite{AB} which were used
in the proof of Proposition \ref{left Frobenius pair}.

\begin{lem}\label{main result of AB} Let $(\mathcal{X},\mathcal{W})$ be a left Frobenius pair in $\mathscr{A}$. Then the following statements hold.
\begin{enumerate}
\item [\rm{(1)}] \cite[Theorem 1.1]{AB}  For each $C\in\widehat{\mathcal{X}}$, there is an exact sequence
$0\ra Y_{C}\ra X_{C}\ra C\ra0$
with $X_{C}\in \mathcal{X}$ and $Y_{C}\in \widehat{\mathcal{W}}.$

\item [\rm{(2)}] \cite[Lemma 3.1]{AB}  The category $\widehat{\mathcal{X}}$ is closed under extensions.
\item [\rm{(3)}] \cite[Proposition 3.6]{AB}  $\widehat{\mathcal{X}}\cap \mathcal{X}^{\perp}=\widehat{\mathcal{W}}.$
    In particular, $\widehat{\mathcal{W}}$ is closed under extensions.

\item [\rm{(4)}] \cite[Theorem 3.7]{AB}  $\mathcal{X}\cap \mathcal{X}^{\perp}=\mathcal{X}\cap\widehat{\mathcal{W}}=\mathcal{W}.$ In particular, $\mathcal{W}$ is closed under finite direct sums.
\end{enumerate}
\end{lem}

The following result contains Theorem \ref{thm:1.1}(1) in the introduction.
\begin{prop}\label{left Frobenius pair} Let $(T\downarrow\mathscr{A})$ be a comma category. Assume that $T(\mathcal{V})\subseteq \widehat{\mathcal{W}}$ and $(\mathrm{L}_{n}T)\mathcal{Y}=0$ for all $n\geq 1$.
 The following conditions are equivalent:
\begin{enumerate}
 \item [\rm{(1)}] $(\mathcal{X},\mathcal{W})$ and $(\mathcal{Y},\mathcal{V})$ are left Frobenius pairs in $\mathscr{A}$ and $\mathscr{B}$, respectively.

 \item [\rm{(2)}] $(\mathfrak{B}^{\mathcal{X}}_{\mathcal{Y}},\mathfrak{B}^{\mathcal{W}}_{\mathcal{V}})$ is a left Frobenius pair in $(T\downarrow\mathscr{A})$.
 \end{enumerate}
\end{prop}
\begin{proof}
$(1)\Rightarrow (2)$. It is easy to check that $\mathfrak{B}^{\mathcal{X}}_{\mathcal{Y}}$ and $\mathfrak{B}^{\mathcal{W}}_{\mathcal{V}}$ are closed under direct summands.
Let $0\ra\ccc{X_{1}}{Y_{1}}{\varphi_{1}}\ra\ccc{X_{2}}{Y_{2}}{\varphi_{2}}\ra\ccc{X_{3}}{Y_{3}}{\varphi_{3}}\ra0$ be an exact sequence in $(T\downarrow\mathscr{A})$ with $\ccc{X_{3}}{Y_{3}}{\varphi_{3}}\in \mathfrak{B}^{\mathcal{X}}_{\mathcal{Y}}$. Note that $(\mathrm{L}_{1}T)Y_{3}=0$ by hypothesis.
Then we have the following exact commutative diagram in $\mathscr{A}$
$$\xymatrix{&  & &0\ar[d]&\\
 0\ar[r]&T(Y_{1})\ar[r]^{}\ar[d]_{\varphi_{1}}&T(Y_{2})\ar[r]^{}\ar[d]_{\varphi_{2}}&T(Y_{3})\ar[r]\ar[d]_{\varphi_{3}}&0\\
 0\ar[r]&X_{1}\ar[r]^{}\ar[d]&X_{2}\ar[r]\ar[d]&X_{3}\ar[r]\ar[d]&0\\
 0\ar[r]&\mathrm{Coker} \varphi_{1}\ar[r]^{}\ar[d]&\mathrm{Coker} \varphi_{2}\ar[r]\ar[d]&\mathrm{Coker}\varphi_{3}\ar[r]\ar[d]&0\\
 &0&0&\;0.&}$$
By the Snake Lemma, $\ccc{X_{1}}{Y_{1}}{\varphi_{1}}\in \mathfrak{B}^{\mathcal{X}}_{\mathcal{Y}}$ if and only if $\ccc{X_{2}}{Y_{2}}{\varphi_{2}}\in \mathfrak{B}^{\mathcal{X}}_{\mathcal{Y}}$.
Hence $\mathfrak{B}^{\mathcal{X}}_{\mathcal{Y}}$ is closed under extensions and kernels of epimorphisms.

Next, we claim that $\mathfrak{B}^{\mathcal{W}}_{\mathcal{V}}$ is an injective cogenerator for $\mathfrak{B}^{\mathcal{X}}_{\mathcal{Y}}$.
Let $\ccc{X}{Y}{\varphi}\in \mathfrak{B}^{\mathcal{X}}_{\mathcal{Y}}$ and
$\ccc{W}{V}{\psi}\in \mathfrak{B}^{\mathcal{W}}_{\mathcal{V}}$. Then we have an exact sequence
$$0\ra\ccc{T(Y)}{Y}{}\ra\ccc{X}{Y}{\varphi}\ra\ccc{{\mathrm{Coker}} {\varphi}}{0}{}\ra 0,$$
which induces an exact sequence
$$\Ex^{i}_{(T\downarrow\mathscr{A})}(\ccc{\mathrm{Coker} {\varphi}}{0}{},\ccc{W}{V}{\psi})\ra
\Ex^{i}_{(T\downarrow\mathscr{A})}(\ccc{X}{Y}{\varphi},\ccc{W}{V}{\psi})\ra
\Ex^{i}_{(T\downarrow\mathscr{A})}(\ccc{T(Y)}{Y}{},\ccc{W}{V}{\psi})$$ for all $i\geq 1.$
Since $\mathcal{V}$ is an injective cogenerator for $\mathcal{Y},$  we have $$\Ex^{i}_{(T\downarrow\mathscr{A})}(\ccc{T(Y)}{Y}{},\ccc{W}{V}{\psi})\cong\Ex^{i}_{\mathscr{B}}(Y,V)=0$$ for all $i\geq 1$ by Lemma \ref{lem:extension group}.
We need to show that $\Ex^{i}_{(T\downarrow\mathscr{A})}(\ccc{\mathrm{Coker} \varphi}{0}{},\ccc{W}{V}{\psi})=0$ for all $i\geq 1.$
Note that $\ccc{W}{V}{\psi}\in\mathfrak{B}^{\mathcal{W}}_{\mathcal{V}}$. Then we have an exact sequence
$$0\ra\ccc{T(V)}{V}{}\ra\ccc{W}{V}{\psi}\ra\ccc{\mathrm{Coker} \psi}{0}{}\ra 0,$$
which induces an exact sequence
$$\Ex^{i}_{(T\downarrow\mathscr{A})}(\ccc{\mathrm{Coker} \varphi}{0}{},\ccc{T(V)}{V}{})\ra
\Ex^{i}_{(T\downarrow\mathscr{A})}(\ccc{\mathrm{Coker} \varphi}{0}{},\ccc{W}{V}{\psi})\ra
\Ex^{i}_{(T\downarrow\mathscr{A})}(\ccc{\mathrm{Coker} \varphi}{0}{},\ccc{\mathrm{Coker} \psi}{0}{})$$
for all $i\geq 1.$
Since $\mathrm{Coker} \varphi\in \mathcal{X},$ $\mathrm{Coker} \psi\in \mathcal{W}$ and $T(\mathcal{V})\subseteq \widehat{\mathcal{W}}$,
it follows from Lemma \ref{lem:extension group} that
 $$\Ex^{i}_{(T\downarrow\mathscr{A})}(\ccc{\mathrm{Coker} \varphi}{0}{},\ccc{T(V)}{V}{})\cong
\Ex^{i}_{\mathscr{A}}(\mathrm{Coker} \varphi,T(V))=0$$ and
$$\Ex^{i}_{(T\downarrow\mathscr{A})}(\ccc{\mathrm{Coker} \varphi}{0}{},\ccc{\mathrm{Coker} \psi}{0}{})\cong
\Ex^{i}_{\mathscr{A}}(\mathrm{Coker} \varphi,\mathrm{Coker} \psi)=0$$
for all $i\geq 1$.
So $\Ex^{i}_{(T\downarrow\mathscr{A})}(\ccc{\mathrm{Coker} \varphi}{0}{},\ccc{W}{V}{\psi})=0$ for all $i\geq 1$.
Thus $\Ex^{i}_{(T\downarrow\mathscr{A})}(\ccc{X}{Y}{\varphi},\ccc{W}{V}{\psi})=0$ for all $i\geq 1$, and hence $\mathfrak{B}^{\mathcal{W}}_{\mathcal{V}}\subseteq (\mathfrak{B}^{\mathcal{X}}_{\mathcal{Y}})^{\perp}$.

 Since
 $Y$ is in $\mathcal{Y}$, we have an exact sequence $0\ra Y\ra V'\ra Y'\ra 0$ in $\mathscr{B}$ with $V'\in\mathcal{V}$ and $Y'\in\mathcal{Y}$ as $\mathcal{V}$ is a cogenerator for $\mathcal{Y}.$
 Similarly, we also have an exact sequence
 $$0\ra \mathrm{Coker}\varphi\ra W'\ra X'\ra 0$$ in $\mathscr{A}$ with $W'\in \mathcal{W}$ and $X'\in \mathcal{X}.$
 Since $T(\mathcal{V})\subseteq\widehat{\mathcal{W}},$ we have $T(\mathcal{V})\subseteq \mathcal{X}^{\perp}.$
It follows from Lemma \ref{lem:extension group} that $$\Ex^{i}_{(T\downarrow\mathscr{A})}(\ccc{\mathrm{Coker}\varphi}{0}{},\ccc{T(V')}{V'}{})\cong
\Ex^{i}_{\mathscr{A}}(\mathrm{Coker}\varphi,T(V'))=0$$ for all $i\geq 1$.
So $\Hom_{(T\downarrow\mathscr{A})}(\ccc{X}{Y}{\varphi},\ccc{T(V')}{V'}{})\ra
\Hom_{(T\downarrow\mathscr{A})}(\ccc{T(Y)}{Y}{},\ccc{T(V')}{V'}{})$ is an epimorphism.
Hence we have the following exact commutative diagram in $(T\downarrow\mathscr{A})$
$$\xymatrix{&0\ar[d]&0\ar[d]&0\ar[d]&\\
 0\ar[r]&\ccc{T(Y)}{Y}{}\ar[r]^{}\ar[d]_{}&\ccc{X}{Y}{\varphi}\ar[r]^{}\ar[d]_{}&\ccc{\mathrm{Coker}\varphi}{0}{}\ar[r]\ar[d]_{}&0\\
 0\ar[r]&\ccc{T(V')}{V'}{}\ar[r]^{}\ar[d]&\ccc{T(V')}{V'}{}\oplus\ccc{W'}{0}{}\ar[r]^{}\ar[d]&\ccc{W'}{0}{}\ar[r]\ar[d]&0\\
 0\ar[r]&\ccc{T(Y')}{Y'}{}\ar[r]^{}\ar[d]&\ccc{C}{Y'}{\theta}\ar[r]\ar[d]&\ccc{X'}{0}{}\ar[r]\ar[d]&0\\
 &0&0&\;0.&}\eqno{\raisebox{-20.5ex}{(3.1)}}$$
We mention that $W'\in \W$. Then $\ccc{T(V')}{V'}{}\oplus \ccc{W'}{0}{}\in \mathfrak{B}^{\mathcal{W}}_{\mathcal{V}}.$
Since $\ccc{T(Y')}{Y'}{}$ and $\ccc{X'}{0}{}$ are in $\mathfrak{B}^{\mathcal{X}}_{\mathcal{Y}},$ one has $\ccc{C}{Y'}{\theta}\in \mathfrak{B}^{\mathcal{X}}_{\mathcal{Y}}$ as
$\mathfrak{B}^{\mathcal{X}}_{\mathcal{Y}}$ is closed under extensions.
Hence we already have shown that $\mathfrak{B}^{\mathcal{W}}_{\mathcal{V}}$ is an injective cogenerator for $\mathfrak{B}^{\mathcal{X}}_{\mathcal{Y}}$.
So $(\mathfrak{B}^{\mathcal{X}}_{\mathcal{Y}},\mathfrak{B}^{\mathcal{W}}_{\mathcal{V}})$ is a left Frobenius pair in $(T\downarrow\mathscr{A})$, as desired.

$(2)\Rightarrow (1).$ First we show that $\mathcal{X},\mathcal{Y},\mathcal{V}$ and $\mathcal{W}$ are closed under direct summands.
Let $Y_{1}\oplus Y_{2}\in \mathcal{Y}$.
Then $\ccc{T(Y_{1}\oplus Y_{2})}{Y_{1}\oplus Y_{2}}{}\in \mathfrak{B}^{\mathcal{X}}_{\mathcal{Y}}.$
Note that $\ccc{T(Y_{1}\oplus Y_{2})}{Y_{1}\oplus Y_{2}}{}\cong\ccc{T(Y_{1})}{Y_{1}}{}\oplus \ccc{T(Y_{2})}{Y_{2}}{}.$
Since $\mathfrak{B}^{\mathcal{X}}_{\mathcal{Y}}$ is closed under direct summands,
it follows that $\ccc{T(Y_{1})}{Y_{1}}{}$ and $\ccc{T(Y_{2})}{Y_{2}}{}$ are in $\mathfrak{B}^{\mathcal{X}}_{\mathcal{Y}}.$
Thus $Y_{1}$ and $Y_{2}$ are in $\mathcal{Y}.$
Hence $\mathcal{Y}$ is closed under direct summands.
Let $X_{1}\oplus X_{2}\in \mathcal{X}$.
Then $\ccc{X_{1}\oplus X_{2}}{0}{}=\ccc{X_{1}}{0}{}\oplus\ccc{X_{2}}{0}{}\in \mathfrak{B}^{\mathcal{X}}_{\mathcal{Y}}.$
So $\ccc{X_{1}}{0}{}$ and $\ccc{X_{2}}{0}{}$ are in $\mathfrak{B}^{\mathcal{X}}_{\mathcal{Y}}$ as $\mathfrak{B}^{\mathcal{X}}_{\mathcal{Y}}$ is closed under direct summands.
Thus $X_{1}$ and $X_{2}$ are in $\mathcal{X}.$
Hence $\mathcal{X}$ is closed under direct summands.
Similarly, we can show that $\mathcal{W}$ and $\mathcal{V}$ are closed under direct summands.

Let $0\ra X_{1}\ra X_{2}\ra X_{3}\ra 0$ be an exact sequence in $\mathscr{A}$ with
$X_{3}\in \mathcal{X}$.
Then we have an exact sequence $0\ra \ccc{X_{1}}{0}{}\ra\ccc{X_{2}}{0}{}\ra\ccc{X_{3}}{0}{}\ra 0$ with $\ccc{X_{3}}{0}{}\in \mathfrak{B}^{\mathcal{X}}_{\mathcal{Y}}.$
Note that $(\mathfrak{B}^{\mathcal{X}}_{\mathcal{Y}},\mathfrak{B}^{\mathcal{W}}_{\mathcal{V}})$ is a left Frobenius pair in $(T\downarrow\mathscr{A})$. It follows that $\ccc{X_{2}}{0}{}\in \mathfrak{B}^{\mathcal{X}}_{\mathcal{Y}}$ if and only if $\ccc{X_{1}}{0}{}\in \mathfrak{B}^{\mathcal{X}}_{\mathcal{Y}}$ if and only if $X_{2}\in \mathcal{X} $ if and only if $X_{1}\in \mathcal{X}.$
Thus $\mathcal{X}$ is closed under extensions and kernels of epimorphisms.
Let $X\in \mathcal{X}$ and $W\in\mathcal{W}$. Then $\ccc{X}{0}{}\in \mathfrak{B}^{\mathcal{X}}_{\mathcal{Y}}$ and $\ccc{W}{0}{}\in \mathfrak{B}^{\mathcal{W}}_{\mathcal{V}}$. It follows from Lemma \ref{lem:extension group} that $\Ex^{i}_{\mathscr{A}}(X,W)\cong\Ex^{i}_{(T\downarrow\mathscr{A})}(\ccc{X}{0}{},\ccc{W}{0}{})=0$
for all $i\geq 1$.
So $\mathcal{W}\subseteq \mathcal{X}^{\perp}$.
Since $\mathfrak{B}^{\mathcal{W}}_{\mathcal{V}}\subseteq\mathfrak{B}^{\mathcal{X}}_{\mathcal{Y}},$
it is easy to check that $\mathcal{W}\subseteq\mathcal{X}$ and $\mathcal{V}\subseteq\mathcal{Y}.$
For any $X\in \mathcal{X},$ we have an exact sequence
 $$0\ra \ccc{X}{0}{}\ra \ccc{W}{V}{\varphi}\ra \ccc{N}{V}{\psi}\ra 0$$ with $\ccc{W}{V}{\varphi}\in \mathfrak{B}^{\mathcal{W}}_{\mathcal{V}}$ and $\ccc{N}{V}{\psi}\in \mathfrak{B}^{\mathcal{X}}_{\mathcal{Y}}$ as $\mathfrak{B}^{\mathcal{W}}_{\mathcal{V}}$ is an injective cogenerator for $\mathfrak{B}^{\mathcal{X}}_{\mathcal{Y}}$. Thus we have
 the following exact commutative diagram in $\mathscr{A}$

 $$\xymatrix{&&0\ar[d]&0\ar[d]\\0\ar[r]&0\ar[r]\ar[d]&T(V)\ar@{=}[r]\ar[d]^\varphi&T(V)\ar[r]\ar[d]^\psi&0\\
    0\ar[r]&X\ar[r]&W\ar[d]\ar[r]&N\ar[d] \ar[r]&0\\
    &&\mathrm{Coker}\varphi\ar[r]\ar[d]&\mathrm{Coker}\psi\ar[d]\\&&0&\;0.\\}$$
Since $T(\mathcal{V})\subseteq \widehat{\mathcal{W}}$ and $\mathrm{Coker}\varphi\in \mathcal{W},$ the second column in the above diagram splits. It follows that $W\cong T(V)\oplus \mathrm{Coker}\varphi$.
 We mention that  $\mathrm{Coker}\psi$
 is in $\X$ as $\ccc{N}{V}{\psi}\in \mathfrak{B}^{\mathcal{X}}_{\mathcal{Y}}.$
Then $\Ex^{i}_{\mathscr{A}}(\mathrm{Coker}\psi,T(V))=0$ by $\W\subseteq\X^{\perp}.$
It implies that $N\cong T(V)\oplus \mathrm{Coker}\psi$.
Since $\X$ is closed under finite direct sums and $\W\subseteq\X$, it is easy to check that $N\in \widehat{\X}$.
So one deduces that $\ccc{N}{0}{}\in \widehat{\mathfrak{B}^{\mathcal{X}}_{\mathcal{Y}}}$.
Since $(\mathfrak{B}^{\mathcal{X}}_{\mathcal{Y}},\mathfrak{B}^{\mathcal{W}}_{\mathcal{V}})$ is a left Frobenius pair in $(T\downarrow\mathscr{A})$ by assumption,
it follows from Lemma \ref{main result of AB}(1) that there exists an exact sequence
$$0\ra \ccc{W_{1}}{V_{1}}{\theta_{1}}\ra \ccc{X_{1}}{V_{1}}{\varphi_{1}}\ra \ccc{N}{0}{}\ra 0$$
with $\ccc{W_{1}}{V_{1}}{\theta_{1}}\in \widehat{\mathfrak{B}^{\mathcal{W}}_{\mathcal{V}}}$ and $\ccc{X_{1}}{V_{1}}{\varphi_{1}}\in \mathfrak{B}^{\mathcal{X}}_{\mathcal{Y}}.$
Thus we have  the following exact commutative diagram in $\mathscr{A}$
$$\xymatrix{&T(V_{1})\ar@{=}[r]^{}\ar[d]_{\theta_{1}}&T(V_{1})\ar[r]^{}\ar[d]_{\varphi_{1}}&0\ar[r]\ar[d]&0\\
    0\ar[r]&W_{1}\ar[r]^{}&X_{1}\ar[r]^{}&N \ar[r]&0.}$$
By the Snake Lemma, there is an exact sequence
$0\ra \mathrm{Coker} \theta_{1}\ra \mathrm{Coker} \varphi_{1}\ra N\ra 0$ with $\mathrm{Coker} \varphi_{1}\in \X$.
Since $\ccc{W_{1}}{V_{1}}{\theta_{1}}\in \widehat{\mathfrak{B}^{\mathcal{W}}_{\mathcal{V}}}$ and
$(\mathrm{L}_{n}T)\mathcal{Y}=0$ for all $n\geq 1$, it is easy to check that
 $\mathrm{Coker} \theta_{1}\in \widehat{\W}$.
Consider the pullback diagram in $\mathscr{A}$
$$\xymatrix{&&0\ar[d]&0\ar[d]\\
&&\mathrm{Coker} \theta_{1}\ar@{=}[r]\ar[d]&\mathrm{Coker} \theta_{1}\ar[d]\\
0\ar[r]&X\ar@{=}[d]\ar[r]&Z\ar[d]\ar[r]&\mathrm{Coker} \varphi_{1}\ar[d]\ar[r]&0\\
0\ar[r]&X\ar[r]&W\ar[d]\ar[r]&N\ar[d]\ar[r]&0\\
&&0&\;0.\\
}$$
Since $\X$ is closed under extensions, one has $Z\in \X.$
We mention that $\ccc{T(V)}{0}{}$ and $\ccc{\mathrm{Coker} \varphi}{0}{}$ are in $\widehat{\mathfrak{B}^{\mathcal{W}}_{\mathcal{V}}}$. Then $\ccc{W}{0}{}\cong \ccc{T(V)}{0}{}\oplus \ccc{\mathrm{Coker} \varphi}{0}{}\in \widehat{\mathfrak{B}^{\mathcal{W}}_{\mathcal{V}}}$ by Lemma \ref{main result of AB}(3).
On the other hand, the second column in the above diagram implies that there is an exact sequence
$0\ra \ccc{\mathrm{Coker} \theta_{1}}{0}{}\ra \ccc{Z}{0}{}\ra \ccc{W}{0}{}\ra 0$ in $(T\downarrow\mathscr{A})$.
It follows from Lemma \ref{main result of AB}(4) that $\ccc{Z}{0}{}\in \mathfrak{B}^{\mathcal{X}}_{\mathcal{Y}}\cap \widehat{\mathfrak{B}^{\mathcal{W}}_{\mathcal{V}}}=\mathfrak{B}^{\mathcal{W}}_{\mathcal{V}}.$
Hence $Z\in \W.$
Thus the second row in the above diagram implies that $\mathcal{W}$ is a cogenerator for $\mathcal{X}.$ So $(\mathcal{X},\mathcal{W})$ is a left Frobenius pair in $\mathscr{A}$.

Let $0\ra Y_{1}\ra Y_{2}\ra Y_{3}\ra 0$ be an exact sequence in $\mathscr{B}$ with $Y_{3}\in \mathcal{Y}$.
Then we have an exact sequence $$0\ra \ccc{T(Y_{1})}{Y_{1}}{}\ra \ccc{T(Y_{2})}{Y_{2}}{}\ra \ccc{T(Y_{3})}{Y_{3}}{}\ra 0$$ with $\ccc{T(Y_{3})}{Y_{3}}{}\in \mathfrak{B}^{\mathcal{X}}_{\mathcal{Y}}$ as $(\mathrm{L}_{1}T)Y_3=0$ by hypothesis.
Hence $Y_{1}\in \mathcal{Y}$ if and only if $\ccc{T(Y_{1})}{Y_{1}}{}\in \mathfrak{B}^{\mathcal{X}}_{\mathcal{Y}}$
if and only if $\ccc{T(Y_{2})}{Y_{2}}{}\in \mathfrak{B}^{\mathcal{X}}_{\mathcal{Y}}$
if and only if $Y_{2}\in \mathcal{Y}$.
Thus $\mathcal{Y}$ is closed under extensions and kernels of epimorphisms.
For any $Y\in \mathcal{Y}$, we have an exact sequence
$$0\ra \ccc{T(Y)}{Y}{}\ra \ccc{W'}{V'}{\theta'}\ra \ccc{X'}{Y'}{\varphi'}\ra 0$$ with $\ccc{W'}{V'}{\theta'}\in \mathfrak{B}^{\mathcal{W}}_{\mathcal{V}}$ and $\ccc{X'}{Y'}{\varphi'}\in \mathfrak{B}^{\mathcal{X}}_{\mathcal{Y}}$ as $\mathfrak{B}^{\mathcal{W}}_{\mathcal{V}}$ is an injective cogenerator for $\mathfrak{B}^{\mathcal{X}}_{\mathcal{Y}}$. Thus $V'\in \mathcal{V}$ and $Y'\in \mathcal{Y}$.
The exact sequence $0\ra Y\ra V'\ra Y'\ra 0$ in $\mathscr{B}$ implies that $\mathcal{V}$ is a cogenerator for $\mathcal{Y}.$
Let $L\in \mathcal{Y}$ and $J\in\mathcal{V}$.
Then $\ccc{T(L)}{L}{}\in \mathfrak{B}^{\mathcal{X}}_{\mathcal{Y}}$ and
$\ccc{T(J)}{J}{}\in\mathfrak{B}^{\mathcal{W}}_{\mathcal{V}}$. Thus
$\Ex^{i}_{\mathscr{B}}(L,J)\cong\Ex^{i}_{(T\downarrow\mathscr{A})}(\ccc{T(L)}{L}{},\ccc{T(J)}{J}{})
=0$ for all $i\geq 1$ by Lemma \ref{lem:extension group}, and hence $\mathcal{V}\subseteq\mathcal{Y}^{\perp}$.
So $(\mathcal{Y},\mathcal{V})$ is a left Frobenius pair.
  This completes the proof.
\end{proof}

We end this section with the following result, which contains Theorem \ref{thm:1.1}(2) in the introduction.
\begin{prop}\label{thm1} Let $(T\downarrow\mathscr{A})$ be a comma category. Assume that $T(\mathcal{V})\subseteq \widehat{\mathcal{W}}$, $T(\mathcal{Y})\subseteq\widehat{\mathcal{X}}$ and $(\mathrm{L}_{n}T)\mathcal{Y}=0$ for all $n\geq 1$. The following conditions are equivalent:
\begin{enumerate}
 \item [\rm {(1)}] $(\mathcal{X},\mathcal{W})$ and $(\mathcal{Y},\mathcal{V})$ are strong left Frobenius pairs in $\mathscr{A}$ and $\mathscr{B}$, respectively.

 \item [\rm {(2)}] $(\mathfrak{B}^{\mathcal{X}}_{\mathcal{Y}},\mathfrak{B}^{\mathcal{W}}_{\mathcal{V}})$ is a strong left Frobenius pair in $(T\downarrow\mathscr{A})$.
 \end{enumerate}
\end{prop}
\begin{proof}
$(1)\Rightarrow (2)$. By Proposition \ref{left Frobenius pair},
we only need to show that $\mathfrak{B}^{\mathcal{W}}_{\mathcal{V}}$ is a projective generator for $\mathfrak{B}^{\mathcal{X}}_{\mathcal{Y}}$.

Let $\ccc{X}{Y}{\varphi}\in \mathfrak{B}^{\mathcal{X}}_{\mathcal{Y}}$ and
$\ccc{W}{V}{\psi}\in \mathfrak{B}^{\mathcal{W}}_{\mathcal{V}}$. Then we have an exact sequence
$$0\ra\ccc{T(V)}{V}{}\ra\ccc{W}{V}{\psi}\ra\ccc{\mathrm{Coker}\psi}{0}{}\ra 0,$$
which induces the following exact sequence
$$\Ex^{i}_{(T\downarrow\mathscr{A})}(\ccc{\mathrm{Coker}\psi}{0}{},\ccc{X}{Y}{\varphi})\ra
\Ex^{i}_{(T\downarrow\mathscr{A})}(\ccc{W}{V}{\psi},\ccc{X}{Y}{\varphi})\ra
\Ex^{i}_{(T\downarrow\mathscr{A})}(\ccc{T(V)}{V}{},\ccc{X}{Y}{\varphi})$$ for all $i\geq 1.$
 Note that $\Ex^{i}_{(T\downarrow\mathscr{A})}(\ccc{T(V)}{V}{},\ccc{X}{Y}{\varphi})\cong\Ex^{i}_{\mathscr{B}}(V,Y)=0$
 for all $i\geq 1$ by Lemma \ref{lem:extension group}.
We need to show that $$\Ex^{i}_{(T\downarrow\mathscr{A})}(\ccc{\mathrm{Coker}\varphi}{0}{},\ccc{X}{Y}{\varphi})=0$$
for all $i\geq 1.$
Since $\ccc{X}{Y}{\varphi}\in\mathfrak{B}^{\mathcal{X}}_{\mathcal{Y}}$, we have an exact sequence
$$0\ra\ccc{T(Y)}{Y}{}\ra\ccc{X}{Y}{\varphi}\ra\ccc{\mathrm{Coker}\varphi}{0}{}\ra 0,$$
which induces the following exact sequence
$$\Ex^{i}_{(T\downarrow\mathscr{A})}(\ccc{\mathrm{Coker}\psi}{0}{},\ccc{T(Y)}{Y}{})\ra
\Ex^{i}_{(T\downarrow\mathscr{A})}(\ccc{\mathrm{Coker}\psi}{0}{},\ccc{X}{Y}{\varphi})\ra
\Ex^{i}_{(T\downarrow\mathscr{A})}(\ccc{\mathrm{Coker}\psi}{0}{},\ccc{\mathrm{Coker}\varphi}{0}{})$$
for all $i\geq 1$.
Since $\mathrm{Coker}\psi\in\mathcal{W},\mathrm{Coker}\varphi\in\mathcal{X}$, $T(\mathcal{Y})\subseteq\widehat{\mathcal{X}}$ and $\mathcal{W}$ is a projective generator for $\mathcal{X}$, it follows that $$\Ex^{i}_{(T\downarrow\mathscr{A})}(\ccc{\mathrm{Coker}\psi}{0}{},\ccc{T(Y)}{Y}{})\cong
\Ex^{i}_{\mathscr{A}}(\mathrm{Coker}\psi,T(Y))=0$$ and
$$\Ex^{i}_{(T\downarrow\mathscr{A})}(\ccc{\mathrm{Coker}\psi}{0}{},\ccc{\mathrm{Coker}\varphi}{0}{})\cong
\Ex^{i}_{\mathscr{A}}(\mathrm{Coker}\psi,\mathrm{Coker}\varphi)=0$$ for all $i\geq 1$ by Lemma \ref{lem:extension group}.
Thus $\Ex^{i}_{(T\downarrow\mathscr{A})}(\ccc{\mathrm{Coker}\psi}{0}{},\ccc{X}{Y}{\varphi})=0$ for all $i\geq 1$, and hence $\Ex^{i}_{(T\downarrow\mathscr{A})}(\ccc{W}{V}{\psi},\ccc{X}{Y}{\varphi})=0$ for all $i\geq 1$.
So $\mathfrak{B}^{\mathcal{W}}_{\mathcal{V}}\subseteq{^{\perp}(\mathfrak{B}^{\mathcal{X}}_{\mathcal{Y}})}$.

Since
 $Y\in \mathcal{Y}$ and $\mathrm{Coker} \varphi\in \mathcal{X}$, we have an exact sequence $0\ra Y'\ra V'\ra Y\ra 0$ in $\mathscr{B}$ with $V'\in \mathcal{V}$ and $Y'\in\mathcal{Y}$ as $(\mathcal{Y},\mathcal{V})$ is a strong left Frobenius pair. Similarly, we have an exact sequence  $0\ra X'\ra W'\ra \mathrm{Coker}\varphi\ra 0$ in $\mathscr{A}$ with $W'\in \mathcal{W}$ and $X'\in \mathcal{X}.$
Note that $T(\mathcal{Y})\subseteq\widehat{\mathcal{X}}$. Then $\Ex^{i}_{(T\downarrow\mathscr{A})}(\ccc{W'}{0}{},\ccc{T(Y)}{Y}{})\cong
\Ex^{i}_{\mathscr{A}}(W',T(Y))=0$ for all $i\geq 1$ by Lemma \ref{lem:extension group}. Hence we have the following exact commutative diagram in $(T\downarrow\mathscr{A})$
$$\xymatrix{&0\ar[d]&0\ar[d]&0\ar[d]&\\
 0\ar[r]&\ccc{T(Y')}{Y'}{}\ar[r]^{}\ar[d]_{}&\ccc{K}{Y'}{\rho}\ar[r]^{}\ar[d]_{}&\ccc{X'}{0}{}\ar[r]\ar[d]_{}&0\\
 0\ar[r]&\ccc{T(V')}{V'}{}\ar[r]^{}\ar[d]&\ccc{T(V')}{V'}{}\oplus\ccc{W'}{0}{}\ar[r]^{}\ar[d]&\ccc{W'}{0}{}\ar[r]\ar[d]&0\\
 0\ar[r]&\ccc{T(Y)}{Y}{}\ar[r]^{}\ar[d]&\ccc{X}{Y}{\varphi}\ar[r]\ar[d]&\ccc{\mathrm{Coker}\varphi}{0}{}\ar[r]\ar[d]&0\\
 &0&0&\;0.&}$$
Since $W'\in\W$, it is easy to check that $\ccc{T(V')}{V'}{}\oplus \ccc{W'}{0}{}\in \mathfrak{B}^{\mathcal{W}}_{\mathcal{V}}$.
Note that  $\mathfrak{B}^{\mathcal{X}}_{\mathcal{Y}}$ is closed under extensions. Hence the second column in the above diagram implies that $\mathfrak{B}^{\mathcal{W}}_{\mathcal{V}}$ is a projective generator for $\mathfrak{B}^{\mathcal{X}}_{\mathcal{Y}}$.
So $(\mathfrak{B}^{\mathcal{X}}_{\mathcal{Y}},\mathfrak{B}^{\mathcal{W}}_{\mathcal{V}})$ is a strong left Frobenius pair in $(T\downarrow\mathscr{A})$, as desired.

$(2)\Rightarrow (1).$ By Proposition \ref{left Frobenius pair}, it suffices to check that $\mathcal{W}$ is a projective generator for $\mathcal{X}$ and $\mathcal{V}$ is a projective generator for $\mathcal{Y}.$
Let $X\in \mathcal{X}$.
Then there exists an exact sequence $$0\ra \ccc{N}{V}{\theta}\ra \ccc{W}{V}{\varphi}\ra \ccc{X}{0}{}\ra 0$$
with $\ccc{W}{V}{\varphi}\in \mathfrak{B}^{\mathcal{W}}_{\mathcal{V}}$ and $\ccc{N}{V}{\theta}\in \mathfrak{B}^{\mathcal{X}}_{\mathcal{Y}}$
as $\mathfrak{B}^{\mathcal{W}}_{\mathcal{V}}$ is a generator for $\mathfrak{B}^{\mathcal{X}}_{\mathcal{Y}}.$
Hence we have the following commutative diagram in $\mathscr{A}$
$$\xymatrix{&0\ar[d]&0\ar[d]\\&T(V)\ar@{=}[r]\ar[d]&T(V)\ar[r]\ar[d]&0\ar[d]\\
    0\ar[r]&N\ar[r]\ar[d]&W\ar[d]\ar[r]&X \ar[r]&0\\
    {}&\mathrm{Coker}\theta\ar[r]\ar[d]&\mathrm{Coker}\varphi\ar[d]&&\\&0&\;0.\\}$$
Since $T(V)\in \widehat{\mathcal{W}}$ and $\mathrm{Coker}\varphi\in \mathcal{W},$ the second column in the above diagram splits.
Thus $W\cong T(V)\oplus \mathrm{Coker}\varphi$, and hence $W\in \widehat{\mathcal{W}}$ by
Proposition \ref{left Frobenius pair} and Lemma \ref{main result of AB}(3).
So there is an exact sequence $0\ra Y\ra W'\ra W\ra 0$ in $\mathscr{A}$ with $W'\in \mathcal{W}$ and $Y\in \widehat{\mathcal{W}}.$
 Consider the pullback diagram in $\mathscr{A}$
 $$\xymatrix{&&0\ar[d]&0\ar[d]\\
0\ar[r]&Y\ar@{=}[d]\ar[r]&Z\ar[d]\ar[r]&N\ar[d]\ar[r]&0\\
0\ar[r]&Y\ar[r]&W'\ar[d]\ar[r]&W\ar[d]\ar[r]&0\\
&&X\ar[d]\ar@{=}[r]&X\ar[d]&\\
&&0&\;0.\\
}$$
Since $(\X,\W)$ is a left Frobenius pair by Proposition \ref{left Frobenius pair}, it follows that $Z\in \X.$

The exact sequence $0\ra Z\ra W'\ra X\ra 0$ in $\mathscr{A}$ implies that $\mathcal{W}$ is a generator for $\mathcal{X}.$
Note that $\Ex^{i}_{\mathscr{A}}(Q,M)\cong\Ex^{i}_{(T\downarrow\mathscr{A})}(\ccc{Q}{0}{},\ccc{M}{0}{})=0$
for any $Q\in \mathcal{W}, M\in \mathcal{X}$ and all $i\geq 1$ by Lemma \ref{lem:extension group}.
It follows that $\mathcal{W}\subseteq {^{\perp}\mathcal{X}}$.
Thus $\mathcal{W}$ is a projective generator for $\mathcal{X},$
and hence $(\mathcal{X},\mathcal{W})$ is a strong left Frobenius pair.
Similarly, we can show that $(\mathcal{Y},\mathcal{V})$ is a strong left Frobenius pair. This completes the proof.
\end{proof}

\section{Recollements of stable categories of Frobenius pairs}\label{section4}

In this section, we first employ Theorem \ref{thm:1.1} to prove Theorem \ref{thm:1.2} in the introduction, and then apply Theorem \ref{thm:1.2} to complete hereditary cotorsion pairs and Gorenstein projective objects.

\subsection{Proof of Theorem \ref{thm:1.2}}

The strategy of proving Theorem \ref{thm:1.2} is as follows: First, we apply Theorem \ref{thm:1.1}(1) to show that given left Frobenius pairs $(\mathcal{X},\mathcal{W})$ and $(\mathcal{Y},\mathcal{V})$ in $\mathscr{A}$ and $\mathscr{B}$ respectively, there exists a left Frobenius pair $(\mathfrak{B}^{\mathcal{X}}_{\mathcal{Y}},\mathfrak{B}^{\mathcal{W}}_{\mathcal{V}})$
in the comma category $(T\downarrow\mathscr{A})$. Next, we show that the functors $q,Z_{\mathscr{A}},U_{\mathscr{A}},T_{\mathscr{B}},U_{\mathscr{B}}$ in the following recollemet (4.1) which induces by $(T\downarrow \mathscr{A})$ (see Remark \ref{remark:recollement})

$$\scalebox{0.85}{\xymatrixcolsep{3pc}\xymatrix{
  \mathscr{A}\ar[rr]^{Z_{\mathscr{A}}} && (T\downarrow \mathscr{A})\ar[rr]^{U_{\mathscr{B}}} \ar@/^1.8pc/[ll]^{U_{\mathscr{A}}}\ar@/_1.8pc/[ll]_{q} && \mathscr{B} \ar@/^1.8pc/[ll]^{Z_{\mathscr{B}}}\ar@/_1.8pc/[ll]_{T_{\mathscr{B}}}}}\eqno{(4.1)}$$
can be restricted to the stable categories with respect to those left Frobenius pairs above. However, the functor $Z_{\mathscr{B}}$ in the recollement (4.1) cannot be restricted to the stable categories.
Thus the main technical problem we encounter is that how
to construct an appropriate right triangle functor (denoted by $\overline{s}$) from $\overline{\mathcal{Y}}$ to  $\overline{\mathfrak{B}^{\mathcal{X}}_{\mathcal{Y}}}$ satisfying the requirements. To circumvent this problem, we try to construct a functor $s:\mathcal{Y}\ra \overline{\mathfrak{B}^{\mathcal{X}}_{\mathcal{Y}}}$, and then demonstrate that the functor $s$ can induce a fully faithful functor $\overline{s}:\overline{\mathcal{Y}}\ra \overline{\mathfrak{B}^{\mathcal{X}}_{\mathcal{Y}}}$.

Finally, we show that those right triangle functors $\overline{q}$, $\overline{Z_{\mathscr{A}}}$, $\overline{U_{\mathscr{A}}}$, $\overline{T_{\mathscr{B}}}$, $\overline{U_{\mathscr{B}}}$ and $\overline{s}$ above are nice enough to establish the recollement (4.1) of right triangulated categories in Theorem \ref{thm:1.2}. If in addition $(\mathcal{X},\mathcal{W})$ and $(\mathcal{Y},\mathcal{V})$ are strong left Frobenius pairs, then we show that it is in fact a recollement of triangulated categories by Theorem \ref{thm:1.1}(2).

{\bf Proof of Theorem \ref{thm:1.2}}. Assume that $(\mathcal{X},\mathcal{W})$ and $(\mathcal{Y},\mathcal{V})$ are left Frobenius pairs in $\mathscr{A}$ and $\mathscr{B}$, respectively. It follows from Proposition \ref{left Frobenius pair} that
$(\mathfrak{B}^{\mathcal{X}}_{\mathcal{Y}},\mathfrak{B}^{\mathcal{W}}_{\mathcal{V}})$ is a left Frobenius pair in $(T\downarrow\mathscr{A})$.
So $\overline{\mathcal{X}},\overline{\mathcal{Y}}$ and $\overline{\mathfrak{B}^{\mathcal{X}}_{\mathcal{Y}}}$ are right triangulated categories by Lemma \ref{lem:3.4}.

We first construct the functors involved.
If $\ccc{X}{Y}{\varphi}\rightarrow \ccc{X'}{Y'}{\varphi'}$ factors through an object $\ccc{W}{V}{\theta}\in \mathfrak{B}^{\mathcal{W}}_{\mathcal{V}}$, then $\mathrm{Coker}\varphi\rightarrow \mathrm{Coker}\varphi'$ factors through $\mathrm{Coker}\theta.$ It follows from Remark \ref{remark:recollement}(4) that the functor $q$ induces a functor $\overline{q}:\overline{\mathfrak{B}^{\mathcal{X}}_{\mathcal{Y}}}\rightarrow\overline{\mathcal{X}} $ given by $\ccc{X}{Y}{\varphi}\mapsto\mathrm{Coker}\varphi$.

 Similarly, we also have functors $\overline{Z_{\mathscr{A}}}:\overline{\mathcal{X}}\rightarrow \overline{\mathfrak{B}^{\mathcal{X}}_{\mathcal{Y}}}$ given by $X\mapsto\ccc{X}{0}{}$ and $\overline{T_{\mathscr{B}}}:\overline{\mathcal{Y}}\rightarrow \overline{\mathfrak{B}^{\mathcal{X}}_{\mathcal{Y}}}$ given by $Y\mapsto\ccc{T(Y)}{Y}{}$.
 It is easy to check that $\overline{Z_{\mathscr{A}}}$ and $\overline{T_{\mathscr{B}}}$ are fully faithful.

For any $\ccc{X}{Y}{\varphi}\in \mathfrak{B}^{\mathcal{X}}_{\mathcal{Y}},$
 we have an exact sequence $0\ra T(Y)\ra X\ra \mathrm{Coker}\varphi\ra 0$ with $Y\in\mathcal{Y}$ and $\mathrm{Coker}\varphi\in\X.$
Since $T(\mathcal{Y})\subseteq \mathcal{X}$ and $\mathcal{X}$ is closed under extensions, it follows that $X\in\X.$
Thus we have a functor
 $\overline{U_{\mathscr{A}}}:\overline{\mathfrak{B}^{\mathcal{X}}_{\mathcal{Y}}}\rightarrow
\overline{\mathcal{X}}$ given by $\ccc{X}{Y}{\varphi}\mapsto X$.
Similarly, we have a functor $\overline{U_{\mathscr{B}}}:\overline{\mathfrak{B}^{\mathcal{X}}_{\mathcal{Y}}}\rightarrow
\overline{\mathcal{Y}}$ given by $\ccc{X}{Y}{\varphi}\mapsto Y$.

Since $(q,Z_{\mathscr{A}}), (Z_{\mathscr{A}},U_{\mathscr{A}})$ and $(T_{\mathscr{B}},U_{\mathscr{B}})$ are adjoint pairs by Remark \ref{remark:recollement}, we have that $(\overline{q},\overline{Z_{\mathscr{A}}})$,  $(\overline{Z_{\mathscr{A}}},\overline{U_{\mathscr{A}}})$ and $(\overline{T_{\mathscr{B}}},\overline{U_{\mathscr{B}}})$ are adjoint pairs by \cite[Lemma 2.4]{WL} or \cite[Lemma 3.1]{XZZ}.

{\bf Step 1:} We claim that $\overline{q}$, $\overline{Z_{\mathscr{A}}}$, $\overline{U_{\mathscr{A}}}$, $\overline{T_{\mathscr{B}}}$ and $\overline{U_{\mathscr{B}}}$ are right triangle functors.

We first start to show that $\overline{q}$ is a right triangle functor.
Let $$0\ra\ccc{X'}{Y'}{\varphi'}\ra\ccc{X}{Y}{\varphi}\ra\ccc{X''}{Y''}{\varphi''}\ra0$$ be an exact sequence in $\mathfrak{B}^{\mathcal{X}}_{\mathcal{Y}}$.
Then we have an exact sequence $$0\ra \mathrm{Coker}\varphi^{'}\ra\mathrm{Coker}\varphi\ra\mathrm{Coker}\varphi^{''}\ra 0$$ in $\mathcal{X}$ by the Snake Lemma.
Hence it is enough to show that $\overline{q}\Sigma\ra\Sigma\overline{q}$ is a natural isomorphism by Remark \ref{remark:right triangulated categories}.
 Let $\ccc{f}{g}{}:\ccc{X_{1}}{Y_{1}}{\varphi_{1}}\ra \ccc{X_{2}}{Y_{2}}{\varphi_{2}}$ be a morphism in $\mathfrak{B}^{\mathcal{X}}_{\mathcal{Y}}$.
Since
$\mathcal{V}$ is an injective cogenarator for $\mathcal{Y}$,
one has exact sequences $$0\ra Y_{1}\xrightarrow{\alpha_{1}}V_{1}\xrightarrow{\beta_{1}}Y_{1}'\rightarrow 0$$
and
$$0\ra Y_{2}\xrightarrow{\alpha_{2}}V_{2}\xrightarrow{\beta_{2}}Y_{2}'\rightarrow 0$$
in $\mathscr{B}$ with $V_{i}\in \mathcal{V}$ and $Y_{i}'\in \mathcal{Y}$ for $i=1,2.$
Thus we have the following commutative diagram
$$\xymatrix{0\ar[r]&Y_{1}\ar[r]^{\alpha_{1}}\ar[d]_{g}&V_{1}\ar[r]^{\beta_{1}}\ar[d]_{t}&Y_{1}'\ar[r]\ar[d]_{s}&0\\
    0\ar[r]&Y_{2}\ar[r]^{\alpha_{2}}&V_{2}\ar[r]^{\beta_{2}}&Y_{2}' \ar[r]&0}$$
as $\Ex_{\mathscr{B}}^{1}(Y_{1}',V_{2})=0.$
Since $\mathcal{W}$ is an injective cogenerator for $\mathcal{X}$, we also have the following commutative diagram
$$\xymatrix{0\ar[r]&\mathrm{Coker}\varphi_{1}\ar[r]^{i_{1}}\ar[d]_{h}&W_{1}\ar[r]^{}\ar[d]_{k}&X_{1}'\ar[r]\ar[d]_{}&0\\
    0\ar[r]&\mathrm{Coker}\varphi_{2}\ar[r]^{i_{2}}&W_{2}\ar[r]^{}&X_{2}'\ar[r]&0}$$
 with $W_{i}\in \mathcal{W}$ and $X_{i}'\in \mathcal{X}$  for $i=1,2$ by the similar argument above.
Through the construction of diagram (3.1) in the proof of Proposition \ref{left Frobenius pair}, we have the following commutative diagram:
$$
\xymatrix@C=0.01em{
 & &
\ccc{T(Y_{1})}{Y_{1}}{} \ar[rrr]^{}
    \ar[dl]_{}
    \ar'[d]'[dd][ddd]
 & & &
\ccc{X_{1}}{Y_{1}}{\varphi_{1}} \ar[dl]_{}
    \ar[rrr]^{}
    \ar'[d]'[dd][ddd]
 & & &
\ccc{\mathrm{Coker}\varphi_{1}}{0}{} \ar[dl]^{}
    \ar[ddd]^{} \\
 &
\ccc{T(V_{1})}{V_{1}}{} \ar[rrr]_{}
    \ar[ld]_{}\ar'[d][ddd]
 & & &
\ccc{T(V_{1})}{V_{1}}{}\oplus\ccc{W_{1}}{0}{} \ar[rrr]^{}
    \ar[ld]_{}
    \ar'[d][ddd]
 & & &
\ccc{W_{1}}{0}{} \ar[ddd]^{}
    \ar[dl]^{} & \\
\ccc{T(Y_{1}')}{Y_{1}'}{} \ar[rrr]^{} \ar[ddd]_{}
 & & &
\ccc{C_{1}}{Y_{1}'}{\theta_{1}} \ar[rrr]^{}
    \ar[ddd]^{}
 & & &
\ccc{X_{1}'}{0}{} \ar[ddd]^{} & \\
 & &
\ccc{T(Y_{2})}{Y_{2}}{} \ar[dl]_{}
     \ar'[r]'[rr][rrr]
 & & &
\ccc{X_{2}}{Y_{2}}{\varphi_{2}} \ar[dl]^{}
    \ar'[r]'[rr][rrr]
 & & &
\ccc{\mathrm{Coker}\varphi_{2}}{0}{} \ar[dl]^{} \\
 &
\ccc{T(V_{2})}{V_{2}}{} \ar[ld]_{}
    \ar'[rr][rrr]
    & & &
\ccc{T(V_{2})}{V_{2}}{}\oplus\ccc{W_{2}}{0}{} \ar'[rr][rrr]\ar[dl]
 & & &
\ccc{W_{2}}{0}{} \ar[dl]^{}
& \\
\ccc{T(Y_{2}')}{Y_{2}'}{} \ar[rrr]_{}
 & & &
\ccc{C_{2}}{Y_{2}'}{\theta_{2}}
    \ar[rrr]_{}
 & & &
\ccc{X_{2}'}{0}{} &
}$$
By the above diagram, one has $\overline{q}\Sigma\ccc{X_{i}}{Y_{i}}{\varphi_{i}}=\overline{q}\ccc{C_{i}}{Y_{i}'}{\theta_{i}}=\mathrm{Coker}\theta_{i}\cong X_{i}'=\Sigma\overline{q}\ccc{X_{i}}{Y_{i}}{\varphi_{i}}$ for $i=1,2.$
The commutativity of diagram
$$
\xymatrix{\ccc{C_{1}}{Y_{1}'}{\theta_{1}}\ar[r]\ar[d]&\ccc{X_{1}'}{0}{}\ar[d]\\
\ccc{C_{2}}{Y_{2}'}{\theta_{2}}\ar[r]&\ccc{X_{2}'}{0}{}
}$$impies the commutativity of the following diagram
$$
\xymatrix{\overline{q}\ccc{C_{1}}{Y_{1}'}{\theta_{1}}\ar[r]\ar[d]&\overline{q}\ccc{X_{1}'}{0}{}\ar[d]\\
\overline{q}\ccc{C_{2}}{Y_{2}'}{\theta_{2}}\ar[r]&\overline{q}\ccc{X_{2}'}{0}{}
.}$$
It follows that $\overline{q}\Sigma\ra\Sigma\overline{q}$ is a natural isomorphism.
Thus $\overline{q}$ is a right triangle functor.

Similarly, we can show that
$\overline{Z_{\mathscr{A}}}$, $\overline{U_{\mathscr{A}}}$, $\overline{T_{\mathscr{B}}}$ and $\overline{U_{\mathscr{B}}}$ are right triangle functors.

{\bf Step 2:} We claim that $\mathrm{Im}\overline{Z_{\mathscr{A}}}= \mathrm{Ker}\overline{U_{\mathscr{B}}}.$
By construction, we have that $\mathrm{Im}\overline{Z_{\mathscr{A}}}\subseteq \mathrm{Ker}\overline{U_{\mathscr{B}}}$ and $\mathrm{Ker}\overline{U_{\mathscr{B}}}=\{\ccc{X}{Y}{\varphi}\in \mathfrak{B}^{\mathcal{X}}_{\mathcal{Y}}\mid Y\in \mathcal{V}\}$
in $\mathfrak{B}^{\mathcal{X}}_{\mathcal{Y}}.$
Let $\ccc{X}{V}{\varphi}\in \mathrm{Ker}\overline{U_{\mathscr{B}}}.$
Then we have an exact sequence $$0\ra T(V)\ra X\ra \mathrm{Coker}\varphi\ra 0$$
in $\mathscr{A}$. Since $\mathrm{Coker}\varphi\in \mathcal{X}$ and $T(V)\in {\mathcal{W}},$ the above exact sequence splits.
Thus $X\cong T(V)\oplus\mathrm{Coker}\varphi,$ and hence we have
$$\ccc{X}{V}{\varphi}\cong \ccc{T(V)}{V}{}\oplus\ccc{\mathrm{Coker}\varphi}{0}{}=\ccc{\mathrm{Coker}\varphi}{0}{}=\overline{Z_{\mathscr{A}}}(\mathrm{Coker}\varphi)$$
in $\overline{\mathfrak{B}^{\mathcal{X}}_{\mathcal{Y}}}$. So $\mathrm{Im}\overline{Z_{\mathscr{A}}}= \mathrm{Ker}\overline{U_{\mathscr{B}}},$ as desired.

{\bf Step 3:}
We claim that there exists a fully faithful functor $\overline{s}:\overline{\mathcal{Y}}\ra \overline{\mathfrak{B}^{\mathcal{X}}_{\mathcal{Y}}}$ with $\overline{s}(Y)=\ccc{W}{Y}{\sigma}$ for any $Y\in{\mathcal{Y}}$ such that $(\overline{U_{\mathscr{B}}},\overline{s})$ is an adjoint pair between $\overline{\mathcal{Y}}$ and $\overline{\mathfrak{B}^{\mathcal{X}}_{\mathcal{Y}}}$, where $W$ is an object in $\mathcal{W}$ satisfying that there is an exact sequence $0\ra T(Y)\xrightarrow{\sigma} W\xrightarrow{\pi} \mathrm{Coker}\sigma\ra 0$ in $\mathscr{A}$ with $\mathrm{Coker}\sigma\in \mathcal{X}$.

We justify the claim as follows.
Define a functor $s:\mathcal{Y}\ra \overline{\mathfrak{B}^{\mathcal{X}}_{\mathcal{Y}}}$ as follows.
Let $Y\in \mathcal{Y}$. Since $T(\mathcal{Y})\subseteq\mathcal{X}$ and $(\mathcal{X},\mathcal{W})$ is a left Frobenius pair, one has an exact sequence
$$0\ra T(Y)\xrightarrow{\sigma} W\xrightarrow{\pi} \mathrm{Coker}\sigma\ra 0$$ in $\mathscr{A}$ with $W\in \mathcal{W}$ and $\mathrm{Coker}\sigma\in \mathcal{X}.$ We define $s(Y)=\ccc{W}{Y}{\sigma}$. One can see that it is well-defined from the argument below, by taking $g=\mathrm{Id}_{Y}$.
Let $g:Y\ra Y'$ be a morphism in $\mathcal{Y}$ such that $$0\ra T(Y')\xrightarrow{\sigma'} W'\xrightarrow{\pi'} \mathrm{Coker}\sigma'\ra 0$$
is an exact sequence in $\mathscr{A}$ with $W'\in \mathcal{W}$ and $\mathrm{Coker}\sigma'\in \mathcal{X}.$
Since $\Ex^{1}_{\mathscr{A}}(\mathrm{Coker}\sigma,W')=0$, we have the following commutative diagram
$$\xymatrix{0\ar[r]&T(Y)\ar[r]^{\sigma}\ar[d]_{T(g)}&W\ar[r]^{\pi}\ar[d]_{f}&\mathrm{Coker}\sigma\ar[r]\ar[d]&0\\
    0\ar[r]&T(Y')\ar[r]^{\sigma'}&W'\ar[r]^{\pi'}&\mathrm{Coker}\sigma' \ar[r]&0.}$$
  Define $s(g)=\overline{\ccc{f}{g}{}}:\ccc{W}{Y}{\sigma}\ra \ccc{W'}{Y'}{\sigma'}$. We claim that $s(g)$ is well-defined and hence the functor $s:\mathcal{Y}\ra \overline{\mathfrak{B}^{\mathcal{X}}_{\mathcal{Y}}}$ is well-defined.
Indeed, if we have another map $f':W\ra W'$ with
$f'\sigma=\sigma'T(g)$, then $f-f'$ factors through $\mathrm{Coker} \sigma$.
Since $\mathrm{Coker}\sigma\in \mathcal{X}$, we have a monomorphism $\widetilde{\sigma}:\mathrm{Coker}\sigma\ra \widetilde{W}$ with $\widetilde{W}\in \mathcal{W}.$
Then it is easy to see that $\ccc{f}{g}{}-\ccc{f'}{g}{}$ factors through $\ccc{\widetilde{W}}{0}{}.$
Hence $\ccc{f}{g}{}= \ccc{f'}{g}{}$ in $\overline{\mathfrak{B}^{\mathcal{X}}_{\mathcal{Y}}}.$

Next, we show that the functor $s:\mathcal{Y}\ra \overline{\mathfrak{B}^{\mathcal{X}}_{\mathcal{Y}}}$ induces a  fully faithful functor $\overline{s}:\overline{\mathcal{Y}}\ra \overline{\mathfrak{B}^{\mathcal{X}}_{\mathcal{Y}}}.$
For this, assume that $g:Y\ra Y'$ factors through an object $V\in \mathcal{V}$ with $g=g_{2}g_{1}.$
Since $T(V)\in \mathcal{W},$ we have $\Ex^{1}_{\mathscr{A}}(\mathrm{Coker}\sigma,T(V))=0$.
So there is a morphism $\alpha:W\ra T(V)$ with $T(g_{1})=\alpha\sigma.$

$$\xymatrix{0\ar[r]&T(Y)\ar[r]^{\sigma}\ar[d]_{T(g_{1})}&W\ar[r]^{\pi}\ar[dd]_{f}\ar@/^/@{.>}[dl]_{\alpha}&\mathrm{Coker}\sigma\ar[r]\ar@{^{(}->}[d]^{\widetilde{\sigma}}\ar@/^/@{.>}[ddl]_{\widetilde{f}}&0\\
&T(V)\ar[d]_{T(g_{2})}&&\widetilde{W}\ar@/^0.8pc/@{.>}[dl]_{\beta}&\\
    0\ar[r]&T(Y')\ar[r]^{\sigma'}&W'\ar[r]^{\pi'}&\mathrm{Coker}\sigma' \ar[r]&0.}$$
Since $(f-\sigma'T(g_{2})\alpha)\sigma=f\sigma-\sigma'T(g_{2})T(g_{1})=0,$
there exists a morphism $\widetilde{f}:\mathrm{Coker}\sigma\ra W'$ with $f-\sigma'T(g_{2})\alpha=\widetilde{f}\pi.$
Since $\mathrm{Coker}\sigma\in \mathcal{X},$ there is a
monomorphism $\widetilde{\sigma}:\mathrm{Coker}\sigma\ra \widetilde{W}$ with $\widetilde{W}\in \mathcal{W}.$
Then there is a morphism $\beta:\widetilde{W}\ra W'$ such that $\widetilde{f}=\beta\widetilde{\sigma}.$
Hence $\ccc{f}{g}{}=\ccc{\tiny\begin{bmatrix}\sigma'T(g_{2})&\beta\end{bmatrix}}{g_{2}}{}\ccc{\tiny\begin{bmatrix}\alpha\\ \widetilde{\sigma}\pi\end{bmatrix}}{g_{1}}{}.$
Thus $\ccc{f}{g}{}$ factors through $\ccc{T(V)}{V}{}\oplus\ccc{\widetilde{W}}{0}{}$,
and hence $s$ induces a functor $\overline{s}:\overline{\mathcal{Y}}\ra \overline{\mathfrak{B}^{\mathcal{X}}_{\mathcal{Y}}}$ given by $\overline{s}(Y)= \ccc{W}{Y}{\sigma}$ and $\overline{s}(\overline{g})=\overline{\ccc{f}{g}{}}$.
By construction, it is clear that $\overline{s}$ is full.
If $\ccc{f}{g}{}$ factors through an object $\ccc{W}{V}{\theta}\in\mathfrak{B}^{\mathcal{W}}_{\mathcal{V}},$
then $g$ factors through $V.$
Hence $\overline{s}:\overline{\mathcal{Y}}\ra \overline{\mathfrak{B}^{\mathcal{X}}_{\mathcal{Y}}}$ is faithful.

Finally, we prove that $(\overline{U_{\mathscr{B}}},\overline{s})$ is an adjoint pair between $\overline{\mathcal{Y}}$ and $\overline{\mathfrak{B}^{\mathcal{X}}_{\mathcal{Y}}}$.
Let $\ccc{f}{g}{}:\ccc{X}{Y}{\varphi}\ra\ccc{W'}{Y'}{\sigma'}$ be a morphism in $\mathfrak{B}^{\mathcal{X}}_{\mathcal{Y}}$ such that $0\ra T(Y')\ra W'\ra \mathrm{Coker}\sigma'\ra 0$ is an exact sequence in $\mathscr{A}$ with $W'\in \mathcal{W}$
and $\mathrm{Coker}\sigma'\in \mathcal{X}.$
 We can show that $\ccc{f}{g}{}$ factors through an object $\ccc{W}{V}{\theta}\in\mathfrak{B}^{\mathcal{W}}_{\mathcal{V}}$ with $W\in \mathcal{W}$ if and only if
$g:Y\ra Y'$ factors through $V\in \mathcal{V}$ by the similar argument in the previous paragraph.
This implies that the map from
$\Hom_{\overline{\mathfrak{B}^{\mathcal{X}}_{\mathcal{Y}}}}(\ccc{X}{Y}{\varphi},\ccc{W'}{Y'}{\sigma'})$ to $\Hom_{\overline{\mathcal{Y}}}(Y,Y')
$ given by $\overline{\ccc{f}{g}{}}\mapsto\overline{g}$ is well-defined and injective.
It is easy to see that the given map is surjective.
 Hence we have the following isomorphism, which is natural in both positions
 $$\Hom_{\overline{\mathfrak{B}^{\mathcal{X}}_{\mathcal{Y}}}}(\ccc{X}{Y}{\varphi},\ccc{W'}{Y'}{\sigma'})\cong\Hom_{\overline{\mathcal{Y}}}(Y,Y')
,$$
i.e.,
  $(\overline{U_{\mathscr{B}}},\overline{s})$ is an adjoint pair between $\overline{\mathcal{Y}}$ and $\overline{\mathfrak{B}^{\mathcal{X}}_{\mathcal{Y}}}$.

{\bf Step 4:} We claim that $\overline{s}$ is a right triangle functor.
Let
$0\ra Y'\ra Y\ra Y''\ra 0$
 be an exact sequence in $\mathcal{Y}$.
 Since $T(\mathcal{Y})\subseteq\mathcal{X}$ and $\mathcal{W}$ is an injective cogenerator for $\mathcal{X}$, we have exact sequences $$0\ra T(Y')\xrightarrow{\delta'} W'\ra \mathrm{Coker}\delta'\ra 0$$
  and
  $$0\ra T(Y'')\xrightarrow{\delta''} W''\ra \mathrm{Coker}\delta''\ra 0$$ in $\mathcal{X}$ with $W'\in \mathcal{W}$ and $W''\in \mathcal{W}$.
 Since $\Ex^{1}_{\mathscr{A}}(T(Y''),W')=0,$ we have the following exact commutative diagram
$$\xymatrix{&  0\ar[d]& 0\ar[d]&0\ar[d]&\\
 0\ar[r]&T(Y')\ar[r]^{}\ar[d]_{\delta'}&T(Y)\ar[r]^{}\ar[d]_{\delta}&T(Y'')\ar[r]\ar[d]_{\delta''}&0\\
 0\ar[r]&W'\ar[r]^{}\ar[d]&W'\oplus W''\ar[r]\ar[d]&W''\ar[r]\ar[d]&0\\
 0\ar[r]&\mathrm{Coker} \delta'\ar[r]^{}\ar[d]&\mathrm{Coker} \delta\ar[r]\ar[d]&\mathrm{Coker}\delta''\ar[r]\ar[d]&0\\
 &0&0&\;0.&}$$
 Hence we have an exact sequence $0\ra \ccc{W'}{Y'}{\delta'}\ra \ccc{W'\oplus W''}{Y}{\delta}\ra \ccc{W''}{Y''}{\delta''}\ra 0$ in $\mathfrak{B}^{\mathcal{X}}_{\mathcal{Y}}.$
  So it is enough to show that $\Sigma\overline{s}\ra \overline{s}\Sigma$ is a natural isomorphism by Remark \ref{remark:right triangulated categories}.
Let $M\in\mathcal{Y}$.
Since
$\mathcal{V}$ is an injective cogenerator for $\mathcal{Y}$,
we have an exact sequence
$$0\ra M\xrightarrow{\alpha}V\xrightarrow{\beta}M'\ra 0$$ with $V\in \mathcal{V}$ and $M'\in \mathcal{Y}.$
Note that $T(\mathcal{Y})\subseteq \mathcal{X}$ and $\mathcal{W}$ is an injective cogenerator for $\mathcal{X}$.
Then we have an exact sequence
$$0\ra T(M')\xrightarrow{l}W\ra X\ra 0$$
with $W\in \mathcal{W}$ and $X\in \mathcal{X}$.
Thus we have the following exact sequence
$$\xymatrix@C=3em{0\ar[r]&\ccc{T(V)}{M}{T(\alpha)}\ar[r]^(.4){\ccc{\tiny\begin{bmatrix}1\\0\end{bmatrix}}{\alpha}{}}&\ccc{T(V)\oplus W}{V}{\ccc{1}{lT(\beta)}{}}\ar[r]^(.6){\ccc{\tiny\begin{bmatrix}0&1\end{bmatrix}}{\beta}{}}&\ccc{W}{M'}{l}\ar[r]&0
    }$$ in $\mathfrak{B}^{\mathcal{X}}_{\mathcal{Y}}$.
 Note that the sequence $0\longrightarrow T(V)\xrightarrow{\tiny\begin{pmatrix}1\\lT(\beta)\end{pmatrix}} T(V)\oplus W\xrightarrow{\tiny\begin{pmatrix}-lT(\beta)&1\end{pmatrix}}W\longrightarrow 0$ is exact with $W\in \mathcal{W}$.
 Then $\ccc{T(V)\oplus W}{V}{\ccc{1}{lT(\beta)}{}}\in \mathfrak{B}^{\mathcal{W}}_{\mathcal{V}}$.
Since $T(\mathcal{V})\subseteq\mathcal{W}$, it follows from the definition of $\overline{s}$ that $\overline{s}(M)=\ccc{T(V)}{M}{T(\alpha)}.$
Hence $\Sigma\overline{s}(Y)=\Sigma\ccc{T(V)}{M}{T(\alpha)}=\ccc{W}{M'}{l}=\overline{s}\Sigma(M).$ Thus $\Sigma\overline{s}\rightarrow \overline{s}\Sigma$ is a natural isomorphism.
So $\overline{s}$ is a right triangle functor, as desired.

By the above arguments, we have the desired recollement (1.2) of right triangulated categories in Theorem \ref{thm:1.2}.

Finally, we assume that $(\mathcal{X},\mathcal{W})$ and $(\mathcal{Y},\mathcal{V})$ are strong left Frobenius pairs,
then $(\mathfrak{B}^{\mathcal{X}}_{\mathcal{Y}},\mathfrak{B}^{\mathcal{W}}_{\mathcal{V}})$ is a strong left Frobenius pair in $(T\downarrow\mathscr{A})$ by Proposition \ref{thm1}.
 Thus $\overline{\mathcal{X}},$ $\overline{\mathcal{Y}}$ and $\overline{\mathfrak{B}^{\mathcal{X}}_{\mathcal{Y}}}$ are triangulated categories by Remark \ref{remark}.
So (1.2) is also a recollement of triangulated categories by the proof above. This completes the proof. \hfill$\Box$

\subsection{Applications of Theorem \ref{thm:1.2} }

A consequence of Theorem \ref{thm:1.2} is the following corollary.

\begin{cor}\label{main corollary} Assume that $(T\downarrow\mathscr{A})$ is a comma category,  $(\mathcal{C}_{1},\mathcal{C}_{2})$ and $(\mathcal{D}_{1},\mathcal{D}_{2})$ are complete hereditary cotorsion pairs in $\mathscr{A}$ and $\mathscr{B}$, respectively. If $T(\mathcal{D}_{1}\cap \mathcal{D}_{2})\subseteq \mathcal{C}_{1}\cap \mathcal{C}_{2}$, $T(\mathcal{D}_{1})\subseteq \mathcal{C}_{1}$ and  $(\mathrm{L}_{n}T)\mathcal{D}_{1}=0$ for all $n\geq 1$, then we get the following recollement of right triangulated categories

$$\scalebox{0.85}{\xymatrixcolsep{3pc}\xymatrix{
  \mathcal{C}_{1}/(\mathcal{C}_{1}\cap \mathcal{C}_{2})\ar[rr]^{{\overline{Z_{\mathscr{A}}}}} && \mathfrak{B}^{\mathcal{D}_{2}}_{\mathcal{D}_{1}}/\mathfrak{B}^{\mathcal{C}_{1}\cap\mathcal{C}_{2}}_{\mathcal{D}_{1}\cap \mathcal{C}_{1}}\ar[rr]^{\overline{U_{\mathscr{B}}}} \ar@/^1.8pc/[ll]^{\overline{U_{\mathscr{A}}}}\ar@/_1.8pc/[ll]_{\overline{q}} && \mathcal{D}_{1}/(\mathcal{D}_{1}\cap \mathcal{D}_{2}). \ar@/^1.8pc/[ll]^{\overline{s}}\ar@/_1.8pc/[ll]_{\overline{T_{\mathscr{B}}}}}}$$
  If in addition $(\mathcal{C}_{1},\mathcal{C}_{2})$ and $(\mathcal{D}_{1},\mathcal{D}_{2})$ are projective cotorsion pairs in $\mathscr{A}$ and $\mathscr{B}$, respectively, then it is in fact a recollement of triangulated categories.
\end{cor}
\begin{proof}
The result follows from Theorem \ref{thm:1.2} and Lemma \ref{cotorsion pair}.
\end{proof}
Let $\Lambda=\left( \begin{smallmatrix} R & M \\0 & S \\\end{smallmatrix} \right)$ be a triangular matrix Artin algebra, where the bimodule $_{R}M_{S}$ is finitely generated over a pair of Artin algebras $R$ and $S$. By \cite{XZZ}, the $R$-$S$-bimodule $M$ satisfies the condition $(\mathrm{IP})$, that is, $M\otimes_{S}\mathrm{D}(S_{S})$ is an injective left $R$-module and $M_{S}$ is projective, where $\mathrm{D}$ is the duality for algebras.  Recall in \cite{XZZ} that the monomorphism category $\mathcal{M}(R,M,S)$ induced by bimodule $_{R}M_{S}$ is the subcategory of $\Lambda$-mod consisting of $\ccc{X}{Y}{\phi}$ such that $\phi: M\otimes_{S}Y$ is a monic $R$-map.
Note that for any Artin algebra $A$, $(A\text{-}\mathrm{mod},\mathrm{inj}A)$ is a complete hereditary cotorsion pair, where $\mathrm{inj}A$ is the subcategory of injective $A$-modules.

As a consequence of Corollary \ref{main corollary}, we improve \cite[Theorem 1.3]{XZZ}, where the authors showed that the $A$-$B$-bimodule $M$ over a triangular matrix Artin algebra $\Lambda=\left( \begin{smallmatrix} A & M \\0 & B \\\end{smallmatrix} \right)$ satisfying the condition (IP) only induced a recollement of additive categories.

\begin{cor}\label{cor:4.11}
An $R$-$S$-bimodule $M$ satisfying the condition $(\mathrm{IP})$ induces a recollement of right triangulated categories
~$$\scalebox{0.85}{\xymatrixcolsep{3pc}\xymatrix{
  R\text{-}\overline{\mathrm{mod}}\ar[rr]^{\overline{Z_{\mathscr{A}}}} && \overline{\mathcal{M}(R,M,S)}\ar[rr]^{\overline{U_{\mathscr{B}}}} \ar@/^1.6pc/[ll]^{\overline{U_{\mathscr{A}}}}\ar@/_1.6pc/[ll]_{\overline{q}} && S\text{-}\overline{\mathrm{mod}}. \ar@/^1.6pc/[ll]^{\overline{s}}\ar@/_1.6pc/[ll]_{\overline{T_{\mathscr{B}}}} } }$$
If in addition $R$ and $S$ are selfinjective algebras, then it is in fact a recollement of triangulated categories.
\end{cor}

\begin{proof}
Since the $R$-$S$-bimodule $M$ satisfies the condition $(\mathrm{IP})$, that is, $M\otimes_{S}\mathrm{D}(S_{S})$ is an injective left $R$-module and $M_{S}$ is projective, $\mathrm{Tor}_{1}^{S}(M,X)=0$ for any left $S$-module $X.$
  Note that $(R\text{-}\mathrm{mod},\mathrm{inj}R)$ and $(S\text{-}\mathrm{mod},\mathrm{inj}S)$ are complete hereditary cotorsion pairs.
  Thus
  we get the desired recollement of right triangulated categories by Corollary \ref{main corollary}.
  If in addition $R$ and $S$ are selfinjective algebras, then $(R\text{-}\mathrm{mod},\mathrm{inj}R)$ and $(S\text{-}\mathrm{mod},\mathrm{inj}S)$ are projective cotorsion pairs.
  Thus the desired recollement is the recollement of triangulated categories by Corollary \ref{main corollary}.
\end{proof}

As another consequence of  Corollary \ref{main corollary}, we obtain the following recollements of right triangulated categories; one can compare them with \cite[Theorem 4.12]{Zhu}.

\begin{cor}\label{corollary:4.4}
Let $\Lambda=\left( \begin{smallmatrix} R & R \\0 & R \\\end{smallmatrix} \right)$ be a triangular matrix ring. Denote by $\mathcal{F}(R)$ and $\mathcal{GF}(R)$ the classes of flat and Gorenstein flat left $R$-modules, respectively.
Then we have the recollements of right triangulated categories
$$\scalebox{0.85}{\xymatrixcolsep{3pc}\xymatrix{
  \overline{\mathcal{F}(R)}\ar[rr]^{{\overline{Z_{\mathscr{A}}}}} && \overline{\mathcal{F}(\Lambda)}\ar[rr]^{\overline{U_{\mathscr{B}}}} \ar@/^1.8pc/[ll]^{\overline{U_{\mathscr{A}}}}\ar@/_1.8pc/[ll]_{\overline{q}} && \overline{\mathcal{F}(R)} \ar@/^1.8pc/[ll]^{\overline{s}}\ar@/_1.8pc/[ll]_{\overline{T_{\mathscr{B}}}}}}$$
and
  $$\scalebox{0.85}{\xymatrixcolsep{3pc}\xymatrix{
  \overline{\mathcal{GF}(R)}\ar[rr]^{{\overline{Z_{\mathscr{A}}}}} && \overline{\mathcal{GF}(\Lambda)}\ar[rr]^{\overline{U_{\mathscr{B}}}} \ar@/^1.8pc/[ll]^{\overline{U_{\mathscr{A}}}}\ar@/_1.8pc/[ll]_{\overline{q}} && \overline{\mathcal{GF}(R)}. \ar@/^1.8pc/[ll]^{\overline{s}}\ar@/_1.8pc/[ll]_{\overline{T_{\mathscr{B}}}}}}$$
\end{cor}

\begin{proof}
Note that $(\mathcal{F}(R),\mathcal{C}(R))$ and $(\mathcal{GF}(R),\mathcal{GF}(R)^{\perp})$ are complete hereditary cotorsion pairs
with $\mathcal{GF}(R) \cap\mathcal{GF}(R)^{\perp}=\mathcal{F}(R)\cap \mathcal{C}(R)$ (see Section 2.3). Thanks to \cite[Proposition 1.14]{FGR} and \cite[Corollary 4.3]{Mao}, we obtain that
$\ccc{X}{Y}{\varphi}\in{\mathcal{F}(R)\cap\mathcal{C}(R)}$ if and only if $Y\in{\mathcal{F}(R)\cap\mathcal{C}(R)}$ and $\varphi: Y\rightarrow X$ is a monic $R$-map with $\textrm{coker}\varphi\in{\mathcal{F}(R)\cap\mathcal{C}(R)}$.
 Thus we get the desired recollements by Corollary \ref{main corollary} and
  \cite[Theorem 2.3]{Mao2}.
\end{proof}

Recall from \cite[Definition 4.1]{HZ} that the right exact functor $T:\mathscr{B}\rightarrow \mathscr{A}$ between abelian categories with enough projective objects is called \emph{compatible}, if the following two conditions hold:
\begin{enumerate}
\item[(C1)] $T(Q^\bullet)$ is exact for any exact sequence $Q^\bullet$ of projective objects in $\mathscr{B}$.
\item[(C2)] ${\rm Hom}_{\mathscr{A}}(P^\bullet,T(Q))$ is  exact for any complete $\mathscr{A}$-projective resolution $P^\bullet$ and any projective object $Q$ in $\mathscr{B}$.
\end{enumerate}

Finally, we apply Theorem \ref{thm:1.2} to the category of Gorenstein projective objects.

\begin{cor}\label{corollary:GP} Assume that $(T\downarrow\mathscr{A})$ is a comma category and $T$ is compatible. Denote by $\mathcal{GP}_\mathscr{A}$  and $\mathcal{P}_\mathscr{A}$ the subcategories of $\mathscr{A}$ consisting of Gorenstein projective objects and projective objects, respectively.
If $T$ preserves projective objects and Gorenstein projective objects, then we get the following recollement of triangulated categories
$$\scalebox{0.85}{\xymatrixcolsep{3pc}\xymatrix{
  \overline{\mathcal {GP}_\mathscr{A}}\ar[rr]^{{\overline{Z_{\mathscr{A}}}}} && \overline{\mathcal{GP}_{(T\downarrow\mathscr{A})}}\ar[rr]^{\overline{U_{\mathscr{B}}}} \ar@/^1.8pc/[ll]^{\overline{U_{\mathscr{A}}}}\ar@/_1.8pc/[ll]_{\overline{q}} && \overline{\mathcal {GP}_\mathscr{B}} \ar@/^1.8pc/[ll]^{\overline{s}}\ar@/_1.8pc/[ll]_{\overline{T_{\mathscr{B}}}}}}.$$
\end{cor}
\begin{proof}
Note that $(\mathcal{GP}_\mathscr{A},\mathcal{P}_\mathscr{A})$ is a strong left Frobenius pair and its proof is similar to that of Proposition 6.1 in \cite{BMPS}. It follows from \cite[Propositions 2.5 and 4.7]{HZ} that $\mathcal {GP}_{(T\downarrow\mathscr{A})}=\mathfrak{B}^{\mathcal {GP}_\mathscr{A}}_{\mathcal {GP}_\mathscr{B}}.$ Thus the result follows from Theorem \ref{thm:1.2}.
\end{proof}

As an application of Corollary \ref{corollary:GP}, we have the following result; one can compare it with \cite[Corollary 4.8]{Zhu} and \cite[Corollary 4.4]{LZHZ}.

\begin{cor} \label{cor:GP1}
Let $\Lambda=\left( \begin{smallmatrix} R & M \\0 & S \\\end{smallmatrix} \right)$ be a triangular matrix ring with $\mathrm{pd}_{R}M<\infty$ and $\mathrm{fd}M_{S}<\infty$. If $_{R}M\otimes_{S}G\in \mathcal{GP}(R)$ for any $G\in \mathcal{GP}(S),$ then we have the following recollement of triangulated categories
$$\scalebox{0.85}{\xymatrixcolsep{3pc}\xymatrix{
  \overline{\mathcal{GP}(R)}\ar[rr]^{{\overline{Z_{\mathscr{A}}}}} && \overline{\mathcal{GP}(\Lambda)}\ar[rr]^{\overline{U_{\mathscr{B}}}} \ar@/^1.8pc/[ll]^{\overline{U_{\mathscr{A}}}}\ar@/_1.8pc/[ll]_{\overline{q}} && \overline{\mathcal{GP}(S)} \ar@/^1.8pc/[ll]^{\overline{s}}\ar@/_1.8pc/[ll]_{\overline{T_{\mathscr{B}}}}}}.$$
\end{cor}
\begin{proof}
Since $\mathrm{fd}M_{S}<\infty,$ $\mathrm{id}_{S}{\rm Hom}_{\mathbb{Z}}(M,\mathbb{Q}/\mathbb{Z})<\infty.$
So we have the following $$ {\rm Hom}_{\mathbb{Z}}(\mathrm{Tor}^{S}_{i}(M,G),\mathbb{Q}/\mathbb{Z}) \cong\Ex^{i}_{S}(G,{\rm Hom}_{\mathbb{Z}}(M,\mathbb{Q}/\mathbb{Z}))=0$$ for any $G\in \mathcal{GP}(S)$ and all $i\geq 1.$
Thus $\mathrm{Tor}^{S}_{i}(M,G)=0$ for all $i\geq 1.$
For any projective left $S$-module $Q$, there exists a projective left $S$-module $P$ such that $Q\oplus P\cong S^{(I)}$.
Note that $M\otimes_{S}(Q\oplus P)\cong M\otimes _{S}S^{(I)}=M^{(I)}$. It follows that $\mathrm{pd}_{R}(M\otimes_{S}Q)\leq\mathrm{pd}_{R}M<\infty.$
Since $M\otimes_{S}Q\in \mathcal{GP}(R)$ by hypothesis,
$M\otimes_{S}Q\in \mathcal{P}(R)$ by \cite[Proposition 10.2.3]{EJ2}.
So the corollary follows from Corollary \ref{corollary:GP}.
\end{proof}

As another application of Corollary \ref{corollary:GP}, we have the following corollary.

\begin{cor}\label{corollary:4.8} Let $(e\downarrow\mathscr{A})$ be a comma category in Example \ref{exact}(4). Then we get the following recollement of triangulated categories
$$\scalebox{0.85}{\xymatrixcolsep{3pc}\xymatrix{
  \overline{\mathcal {GP}_\mathscr{A}}\ar[rr]^{{\overline{Z_{\mathscr{A}}}}} && \overline{\mathcal{GP}_{(e\downarrow\mathscr{A})}}\ar[rr]^{\overline{U_{\mathscr{B}}}} \ar@/^1.8pc/[ll]^{\overline{U_{\mathscr{A}}}}\ar@/_1.8pc/[ll]_{\overline{q}} && \overline{\mathcal{GP}_{\mathrm{Ch}(R)}} \ar@/^1.8pc/[ll]^{\overline{s}}\ar@/_1.8pc/[ll]_{\overline{T_{\mathscr{B}}}}}}.$$
\end{cor}
\begin{proof}
 Note that a complex $X^{\bullet}$ in $\mathrm{Ch}(R)$ is Gorenstein projective if and only if $X^{n}$ is a Gorenstein projective module for all $n\in \mathbb{Z}$ by \cite[Theorem 2.2]{Yang}. Thanks to \cite[Corollary 4.9]{HZ}, we obtain that
 $\left(
                                                                                        \begin{smallmatrix}
                                                                                          X \\
                                                                                          Y^\bullet \\
                                                                                        \end{smallmatrix}
                                                                                      \right)_\varphi$ is a Gorenstein projective object in $(e\downarrow\mathscr{A})$ if and only if $Y^\bullet$  is a Gorenstein projective object in ${\rm Ch}(R)$ and $\varphi:Y^{0}\to X$ is a monic $R$-map with a Gorenstein projective cokernel.
 Since $e: \mathscr{B}\rightarrow\mathscr{A}$ via $C^\bullet\mapsto C^0$ for any $C^\bullet\in{\mathscr{B}}$ is an exact functor,
 the result follows from Corollary \ref{corollary:GP}.
\end{proof}

{\bf Acknowledgements.} The research was supported by the National Natural Science Foundation of China (Grant Nos. 12171206 and 12201223). Also, J.S. Hu thanks the Natural Science Foundation of Jiangsu Province (Grant No. BK20211358) and Jiangsu 333 Project,  and  Y.J. Ma thanks Youth Foundation of Lanzhou Jiaotong University (Grant No. 2023023) and Youth Science and Technology Foundation of Gansu Province (Grant No. 23JRRA866). The authors are grateful to
the referees for reading the paper carefully and for many suggestions on mathematics and English expressions.

\bigskip

{\Large{\textbf{Declarations}}}

\bigskip
\textbf{Conflict of interest} The authors declare no conflict of interest.

\bigskip
\textbf{Yajun Ma}\\
School of Mathematics and Physics, Lanzhou Jiaotong University, Lanzhou 730070, China.\\
E-mail: \textsf{13919042158@163.com}\\[1mm]
\textbf{Dandan Sun}\\
School of Economics, Jiangsu University of Technology,
Changzhou 213001, China.\\
E-mail: \textsf{13515251658@163.com}\\[1mm]
\textbf{Rongmin Zhu}\\
School of Mathematical
Sciences, Huaqiao University, Quanzhou 362021, China.\\
E-mail: \textsf{rongminzhu@hotmail.com}\\[1mm]
\textbf{Jiangsheng Hu}\\
School of Mathematics, Hangzhou Normal University, Hangzhou 311121, China;\\
Department of Mathematics, Jiangsu University of Technology,
Changzhou 213001, China.\\
E-mail: \textsf{jiangshenghu@hotmail.com}\\[1mm]

\end{document}